\documentclass[english,12pt,a4paper]{article}
\usepackage{amsmath,amssymb,amsthm} 

\usepackage[T1]{fontenc}
\usepackage[latin1]{inputenc}
\usepackage{graphicx}
\usepackage[scanall]{psfrag}
\usepackage{comment}
\usepackage{color}
\usepackage{setspace}

\theoremstyle{plain}
\newtheorem{theorem}{Theorem}

\newtheorem{proposition}[theorem]{Proposition}

\newtheorem{remark}[theorem]{Remark}

\numberwithin{theorem}{section}
\numberwithin{equation}{section}

\newcommand{\R}{{\mathbb{R}}}
\newcommand{\rplus}{{\mathbb{R}^{+}}}
\newcommand{\rpluscl}{\overline{\mathbb{R}}^{+}}

\newcommand{\T}{\mathbb{T}}

\newcommand{\clos}[1]{\overline{#1}}

\newcommand{\sbm}[1]{\left[\begin{smallmatrix} #1
             \end{smallmatrix}\right]}

\newcommand{\OmitThis}[1]{}
\newcommand{\ResearchNote}[1]{}

\newcommand{\ProofNote}[1]{}

\newcommand{\closure}[1]{\overline{#1}}
\newcommand{\abs}[1]{\left \vert #1 \right \vert}

\newcommand{\BLO}{\mathcal L}
\newcommand{\norm}[1]{\|{#1}\|}
\newcommand{\rst}[1]{\big{|}_{#1}}

\newcommand{\bg}{\mathbf \gamma}
\newcommand{\bt}{\mathbf t}
\newcommand{\bn}{\mathbf n}
\newcommand{\bb}{\mathbf b}
\newcommand{\bnu}{\mathbf \nu}

\newcommand{\br}{\mathbf r}
\newcommand{\bi}{\mathbf i}
\newcommand{\bj}{\mathbf j}
\newcommand{\bk}{\mathbf k}

\begin{document}

\title{Webster's equation \\ with curvature and dissipation}
\author{Teemu Lukkari and Jarmo Malinen}

\bibliographystyle{plain}

\maketitle

\begin{abstract}
  Wave propagation in curved tubular domains is considered. \linebreak
  A general version of Webster's equation is derived from the
  scattering passive wave equation. More precisely, it is shown that
  planar averages of a sufficiently smooth solution of the wave
  equation satisfy the corresponding Webster's equation when the
  latter includes additional control signals determined by the
  solution.
\end{abstract}


vocal tract.



\section{Introduction}

We study wave propagation in a narrow but long, tubular domain $\Omega
\subset \R^3$ of finite length whose cross-sections are circular and
of varying area.  In this case, the wave equation in $\Omega$, the
topmost equation in \eqref{IntroWaveEq} below, has a classical
approximation depending on a single spatial variable in the long
direction of $\Omega$. The approximation is known as \emph{Webster's
  equation}, which is the topmost equation in \eqref{IntroWebstersEq}
below. The geometry of $\Omega$ is represented by the \emph{area
  function} $A(\cdot)$ whose values are cross-sectional areas of
$\Omega$.  

We represent a tubular domain $\Omega \subset \R^3$ as follows.  The
center line of the tube can be thought of as a smooth curve $\bg$ (of
unit length) parametrised by arc length $s$. We assume that the
cross-section of $\Omega$, perpendicular to the tangent of $\bg$ at
the point $\bg(s)$, is the circular disk $\Gamma(s)$ with the radius
$R(s)$. The boundary of $\Omega$ then consists of the \emph{ends} of
the tube, $\Gamma(0)$ and $\Gamma(1)$, and the \emph{wall} $\Gamma$ of
the tube. With this notation, lossless acoustic wave propagation in
$\Omega$ can be modelled by the wave equation, written for the
\emph{(perturbation) velocity potential} $\phi: \Omega \times \rpluscl 
\to \R$
\begin{equation}\label{IntroWaveEq}
\begin{cases}
  &  \phi_{tt}(\br, t) = c^2 \Delta \phi(\br, t) \quad
  \text{ for } \br \in \Omega \text{ and } t \in \rplus, \\
  &   c \frac{\partial \phi}{\partial
    \bnu}(\br, t) + \phi_t(\br, t) = 2 \sqrt{\tfrac{c}{\rho A(0)}} \, u(\br, t) \quad
  \text{ for } \br \in \Gamma(0) \text{ and }  t \in \rplus, \\
  &  c \frac{\partial \phi}{\partial
    \bnu} (\br, t) - \phi_t(\br, t)  =  2 \sqrt{\tfrac{c}{\rho A(0)}} \, y(\br, t)  \quad
  \text{ for }  \br \in \Gamma(0) \text{ and }  t \in \rplus, \\
  &  \phi(\br, t)  = 0 \quad \text{ for } \br \in \Gamma(1) \text{ and }
  t \in \rplus, \\
  &  \frac{\partial \phi}{\partial \bnu}(\br, t)  = 0 \quad \text{ for } \br
  \in \Gamma, \text{ and }
  t \in \rplus 
\end{cases}
\end{equation}
where $\rplus = (0,\infty)$, $\rpluscl = [0,\infty)$, and $\bnu$
  denotes the unit exterior normal vector field on $\Gamma$. The sound
  speed in and the density of the medium are denoted by $c$ and
  $\rho$, respectively.  The Dirichlet condition on $\Gamma(1)$
  represents an open end, and the Neumann condition on $\Gamma$
  represents a hard reflective surface.  The control (i.e., the input)
  $u(\br, t)$ and the observation (i.e., the output) $y(\br, t)$ are
  given in \emph{scattering form} in \eqref{IntroWaveEq} where the
  physical dimension of both signals is power per unit area.\footnote{
    Another typical way of writing the control and observation for
    \eqref{IntroWaveEq} is by using the impedance/admittance boundary
    conditions $\frac{\partial \phi}{\partial \bnu} (\br, t) = v(\br,
    t)$ and $\rho \phi_t(\br, t) = p(\br, t)$ on the active part
    $\Gamma(0)$ of the boundary $\partial \Omega$ where $v$ and $p$
    are the (perturbation) velocity and (acoustic) pressure,
    respectively.  The interplay between impedance and scattering
    forms is discussed in a general setting in, e.g.,
    \cite{M-S-W:HTCCS,M-S:CBCS,M-S:IPCBCS}.}

The corresponding lossless Webster's equation is given by
\begin{equation} \label{IntroWebstersEq}
  \begin{cases}
    & \psi_{tt} = \frac{c(s)^{2}}{A(s)} \frac{\partial}{\partial s}
    \left ( A(s) \frac{\partial \psi }{\partial s} \right ) 
   \hfill \text{ for }   s \in (0,1) \text{ and } t \in \rplus, \\
    & - c \psi_s(0,t) + \psi_t(0,t) = 2 \sqrt{\frac{c}{\rho
        A(0)}} \, \tilde u(t) \quad
    \text{ for } t \in \rplus, \\
    & - c \psi_s(0,t) - \psi_t(0,t) = 2 \sqrt{\frac{c}{\rho
        A(0)}} \, \tilde y(t) \quad
    \text{ for } t \in \rplus, \\
    & \psi(1,t) = 0 \quad  \text{ for } t \in \rplus
\end{cases}
\end{equation}
where $A(s)$ is the area of the cross-section $\Gamma(s)$, and the
solution $\psi:[0,1]\times \rpluscl \to \R$ is \emph{Webster's
  velocity potential}.  Obviously, Webster's equation
\eqref{IntroWebstersEq} is a computationally and mathematically
simpler model than \eqref{IntroWaveEq} for longitudinal wave
propagation for wavelengths that are long compared to $R(s)$.  More
precisely, the solution $\psi$ to \eqref{IntroWebstersEq} is expected
to approximate the averages
\begin{equation} \label{AveragedWaveEqSol}
 \bar \phi(s,t) : = \frac{1}{A(s)} {\int_{ \Gamma(s) }{\phi d A}} \quad \text{ for } \quad 
s \in (0,1) \quad \text{ and }  \quad  t \in \rpluscl 
\end{equation}
of the velocity potential $\phi$ given by \eqref{IntroWaveEq}. 

\begin{figure}
  \centering  
  \includegraphics[height=8cm]{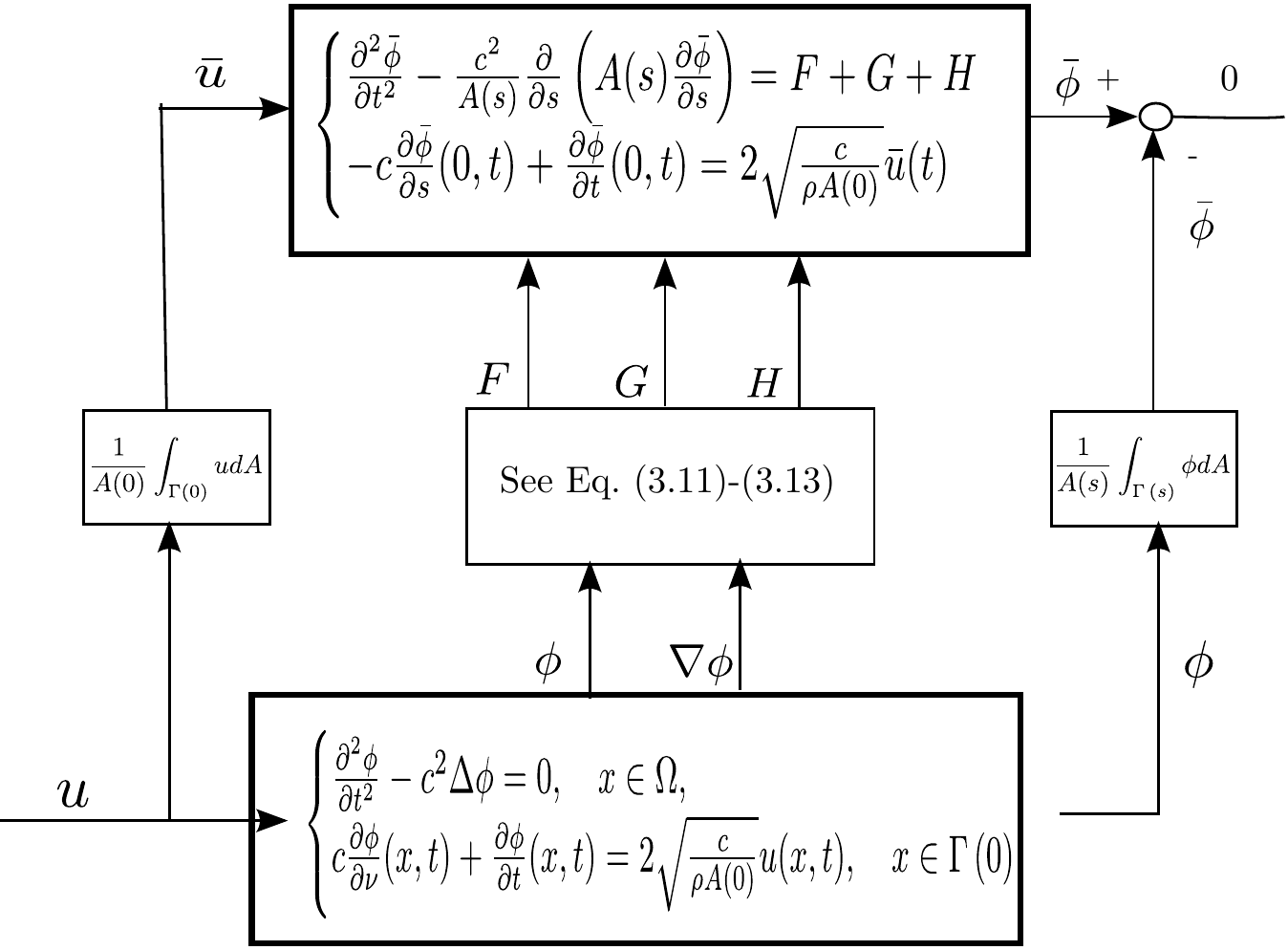}
  \caption{Feedforward coupling describing the main results of this
    article. The equations in the blocks are as they appear in the
    lossless case and without curvature.  The functions $\bar \phi$
    and $\phi$ are related by \eqref{AveragedWaveEqSol}.}
  \label{Fig1}
\end{figure}

The purpose of this article is to derive Webster's equation in a
general and mathematically rigorous way.  The ultimate goal of this
article and its companion articles \cite{A-L-M:AWGIDDS,L-M:PEEWEWP} is
to write an estimate for difference between the solution $\psi$ of
Webster's model \eqref{IntroWebstersEq} and the planar averages $\bar
\phi$ of $\phi$; we call this difference $\psi - \bar \phi$ the
\emph{tracking error}.  All this can be summarised with the aid of
Fig.~\ref{Fig1} and Fig.~\ref{Fig2} as follows:
\begin{enumerate}
\item By Theorem~\ref{MainTheorem1} and the solvability of the
  scattering wave equation model on $\Omega$ by \cite[Theorem~5.1 and
    Corollary~5.2]{A-L-M:AWGIDDS}, all the signals in Fig.~\ref{Fig1}
  are well-defined functions. The average $\bar\phi$ satisfies
  nonhomogeneous Webster's equation \eqref{WebsterEq1} with the load
  term $F + G + H$ given by \eqref{ControlTermsF} --
  \eqref{ControlTermsH}. Thus, the tracking error vanishes in
  Fig.~\ref{Fig1} because this additional load term has been included.
\item 
  In the absense of the load term $F + G + H$ in the true, homogenous
  Webster's equation shown in Fig.~\ref{Fig2}, the tracking error can still be controlled by
  $F$, $G$, and $H$.  This is carried out in \cite{L-M:PEEWEWP} by
  refining the passivity argument of \cite[Section~4]{A-L-M:AWGIDDS}
  to include the required non-boundary control term in the energy
  estimates for Webster's model.
\item The load terms $F$, $G$, and $H$, and hence, the tracking error
  can be \emph{a posteriori} estimated by $\phi$ as shown in
  \cite{L-M:PEEWEWP}.
\end{enumerate}
We conclude that the solution of the homogenous Webster's equation
$\psi$ in \eqref{IntroWebstersEq} approximates well the planar
averages $\bar \phi$ in \eqref{AveragedWaveEqSol} for solutions $\phi$
of \eqref{IntroWebstersEq} for which $F + G + H$ is insignificant, and
the longitudinal wave propagation in $\Omega$ dominates the
transversal modes.

\begin{figure}
  \centering  
  \includegraphics[height=8cm]{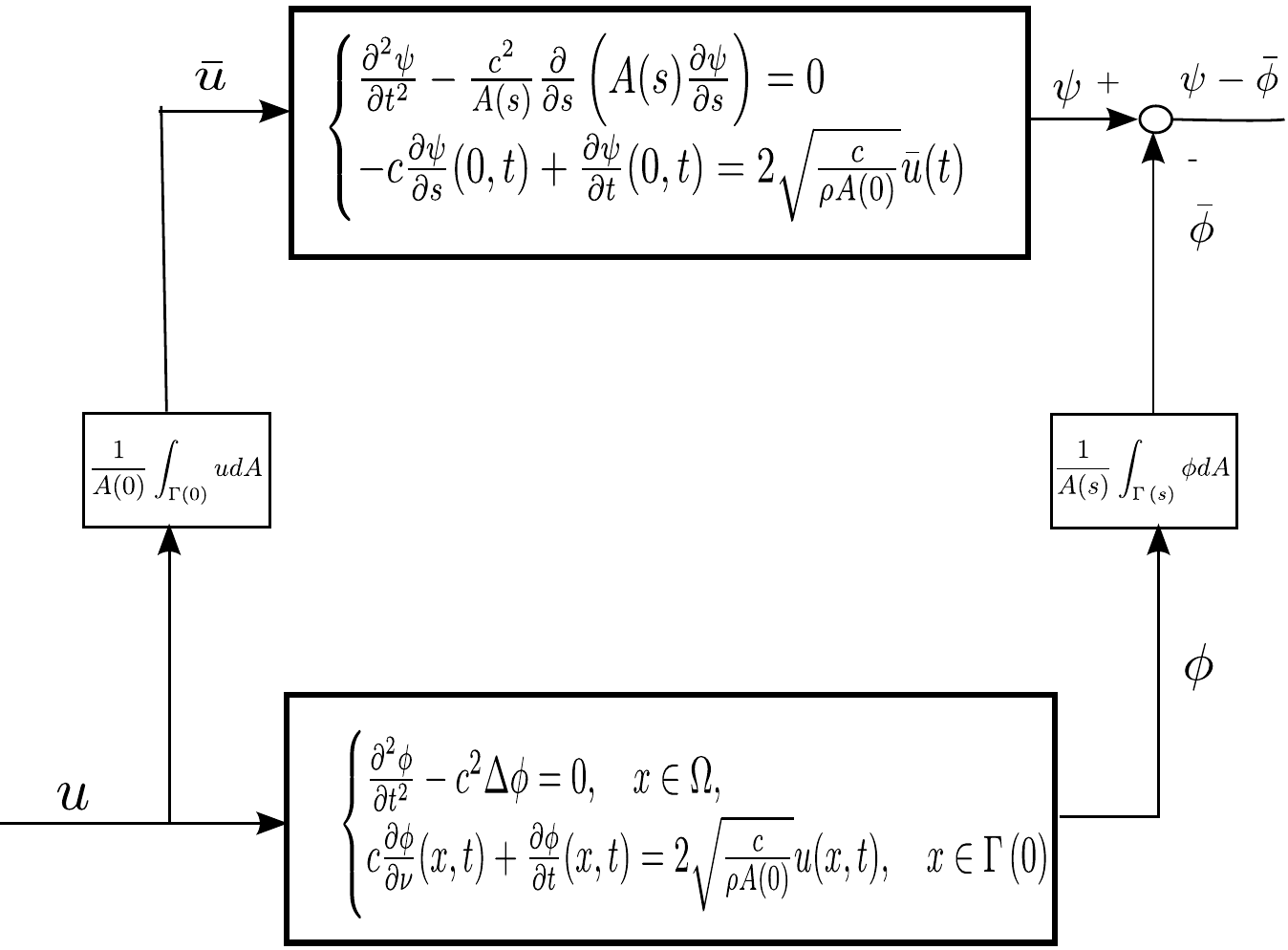}
  \caption{Feedforward coupling describing the tracking error $\psi -
    \bar \phi$. }
  \label{Fig2}
\end{figure}

All these results are shown in the context of a more general wave
equation model \eqref{WaveEq}---\eqref{WaveEqObs} than the lossless
models given by \eqref{IntroWaveEq} which leads to
\eqref{IntroWebstersEq} in the case of straight tubular domains
$\Omega$. Boundary dissipation is allowed through the lateral wall
$\Gamma$ by the boundary condition \eqref{DampingBndry}, and this
leads to a dissipation term in Webster's equation
\eqref{WebsterHorn}. The curvature of $\Omega$ is taken into account
in Webster's equation by introducing a correction factor to the speed
of sound; see \eqref{VariableSoundSpeedDef} and also
\cite[Section~6]{SR:WHER}.  The full results require that the torsion
of the center line $\bg$ of $\Omega$ vanishes, and that $A'(0) =
\kappa(0) = 0$ where $\kappa$ is the curvature of $\bg$.

Our interest in Webster's equation stems from the fact that it
provides a simple and tractable model of the acoustics of the human
vocal tract as it appears during a vowel utterance. This equation (as
well as the wave equation) can be used as a part of a computational
physics model; see, e.g.,
\cite{A-A-M-M-V:MLBVFVTOPI,A-H-K-M-P-S-V:HFAVFFCVTR,
  C-K:TVINAS,GF:ATSP,H-L-M-P:WFWE} and the references therein. The
other parts are a mechanical model for the glottis and an exterior
space model outside the mouth; see
\cite{A-A:LOGMNFMCAL,A-A-M-M-V:MLBVFVTOPI}. Following these ideas and
motivations, a boundary control action is used to represent the
acoustic field produced by the glottis source.  Further applications
of Webster's equation include modelling of water waves in tapered
channels, acoustic design of exhaust pipes and jet engines for
controlling noise, vibration, and performance as well as construction
of instruments such as loudspeakers and horns
\cite[p.~402--405]{H-S:FDHL}.

Early work concerning Webster's equation can be found in
\cite{EE:CSWHE,VS:GPWHT,VS:NFH,AW:AITHP}. Webster's original work
\cite{AW:AITHP} was published in 1919, but the model itself has a
longer history spanning over 200 years and starting from the works of
D.~Bernoulli, Euler, and Lagrange. A selection of more modern
approaches is provided by
\cite{L-L:AMAEMAI,L-L:AMAEMAII,N-T:APDVCS,SR:STSVCALDF,SR:WHER,R-E:NCBMSFESSPLFD}. Moreover,
there exist a wide and relevant PDE-related literature on the damped
wave equation in 1D spatial domain where questions such as
distribution of resonance values, eigenfunction expansions,
exponential stabilisation, and energy decay have been treated: see,
e.g., \cite{G-H:DSPE} which contains a multitude of references in
various directions.

Shared features of the modern treatments of Webster's model (such as
given in \cite[Section~8.1]{R-H:IA}) are the time-harmonic Ansatz and
asymptotic expansions for the velocity potential function $\phi$
satisfying the wave equation (either with or without an underlying
advection). The time-harmonic Ansatz leads to the corresponding
time-invariant form of the model, analogous to the Helmholtz
equation. The resonance structure (i.e., the formants) of the model is
obtained from the resulting eigenvalue problem; see, e.g., asymptotic
spectra of Neumann--Laplacian on shrinking tubular domains in
\cite{KZ:CSMSCG,RS:STRIP}.  Our complementary approach is to look
directly for equations that are satisfied by planar averages $\bar
\phi$ of $\phi$ (see \eqref{AveragedWaveEqSol} below) which is likely
how Webster's equation was first discovered.  This results in
\eqref{WebsterEq1} which is satisfied by the averages $\bar \phi$.  If
the three load terms given by
\eqref{ControlTermsF}--\eqref{ControlTermsH} are neglected from
\eqref{WebsterEq1}, Webster's equation is obtained in a generalised
form.
  
The article is organised as follows: In Section~\ref{BackgroundSec},
we discuss the regularity of the solution $\phi$ of the wave equation
as well as details of the coordinate system for the tubular geometry
described above.  The main result, Theorem \ref{MainTheorem1} is given
without proof in Section \ref{ResultSec}. In Section \ref{WebsterSec},
we present the derivation of Webster's equation with curvature and
dissipation. Also the boundary conditions for Webster's equation are
derived from those of the wave equation. Section \ref{ProofSection}
contains the rigorous proof of Theorem \ref{MainTheorem1}.


\section{\label{BackgroundSec} Background}

\subsubsection*{Regularity and solvability of the wave equation}

In order to carry out the arguments leading to Webster's model for
approximating the planar averages $\bar \phi$, we need have some
assumptions on the regularity of $\phi$ that solves the original wave
equation model. Throughout this article, we require that the solution
$\phi$ of the wave equation has the following regularity
properties\footnote{Even if we write the spatial variables before the
  time variable in functions such as $\phi = \phi(\br,t)$, we write
  the time variable first for the function spaces such as
  $C^1(\rpluscl; H^1(\Omega))$.}:
\begin{equation} \label{WaveEqRegularityAssumptions}
  \begin{aligned} 
    \phi & \in C^1(\rpluscl; H^1(\Omega)) \cap C^2(\rpluscl; L^2(\Omega)), 
    \quad  \frac{\partial \phi}{\partial \nu} \in C(\rpluscl;L^2(\partial \Omega)) \\
    \nabla  \phi & \in C^1(\rpluscl;L^2(\Omega;\R^3)), \quad  \text{ and }   \quad 
    \Delta \phi   \in C(\rpluscl;L^2(\Omega)).
  \end{aligned}  
\end{equation}
We remark that even the method of asymptotic expansions requires a
nontrivial implicit regularity assumption so that the expansion
exists; see \cite[p.~1985]{SR:WHER}.

It is not unreasonable to require \eqref{WaveEqRegularityAssumptions}
from the unique solution of the wave propagation model appearing the
wave equation part in Fig.~\ref{Fig1} and given by
\begin{equation}\label{WaveEq}
\begin{cases}
  & \phi_{tt}(\br, t) = c^2 \Delta \phi(\br, t) \quad \text{ for } \br
  \in \Omega \text{ and } t \in \rplus, \\ & c \frac{\partial
    \phi}{\partial \bnu}(\br, t) + \phi_t(\br, t) = 2
  \sqrt{\tfrac{c}{\rho A(0)}} \, u(\br, t) \quad \text{ for } \br \in
  \Gamma(0) \text{ and } t \in \rplus, \\ & \phi(\br, t) = 0 \quad
  \text{ for } \br \in \Gamma(1) \text{ and } t \in \rplus, \\ &
  \frac{\partial \phi}{\partial \bnu}(\br, t) + \alpha \phi_t(\br, t)
  = 0 \quad \text{ for } \br \in \Gamma, \text{ and } t \in \rplus,
  \text{ and } \\ & \phi(\br,0) = \phi_0(\br), \quad \rho
  \phi_t(\br,0) = p_0(\br) \quad \text{ for } \br \in \Omega,
\end{cases}
\end{equation}
together with the observation $y$ defined by
\begin{equation}\label{WaveEqObs}
    c \frac{\partial \phi}{\partial
    \bnu} (\br, t) - \phi_t(\br, t)  =  2 \sqrt{\tfrac{c}{\rho A(0)}} \, y(\br, t)  \quad
  \text{ for }  \br \in \Gamma(0) \text{ and }  t \in \rplus.
\end{equation}
Indeed, equations \eqref{WaveEq}---\eqref{WaveEqObs} define an
internally well-posed, dissipative dynamical system for $\alpha \geq
0$ as shown in \cite[Theorem~5.1(i)\&(iii)]{A-L-M:AWGIDDS} which is,
in addition, conservative if $\alpha = 0$ by
\cite[Corollary~5.2]{A-L-M:AWGIDDS}.  Such a system has a unique
solution $\phi$ for twice differentiable input signals and compatible
initial conditions $\sbm{\phi_0 \\ p_0}$ at $t = 0$:
\begin{proposition}
\label{PieniPorpositioRatkeavuudesta} 
  Let $u \in C^2(\rpluscl; L^2(\Gamma(0)))$ and $\phi_0, p_0 \in
  H^1(\Omega)$ satisfy
\begin{equation} \label{PieniPorpositioRatkeavuudestaKaava1}
\begin{aligned} 
  & \Delta \phi_0 \in L^2(\Omega), \quad \phi_0\rst{\Gamma(1)} =
  p_0\rst{\Gamma(1)} = 0, \quad \frac{\partial \phi_0}{\partial
    \nu}\rst{\Gamma(0) \cup \Gamma} \in L^2(\Gamma(0) \cup \Gamma),
  \\ & \rho c \frac{\partial \phi_0}{\partial \nu}\rst{\Gamma(0)} +
  p_0 \rst{\Gamma(0)} = \sqrt{\frac{\rho c}{A(0)}} u(\cdot,0) \text{
    and } \rho \frac{\partial \phi_0}{\partial \nu}\rst{\Gamma} +
  \alpha p_0 \rst{\Gamma} = 0.
\end{aligned}
\end{equation}
  Then there is a unique solution $\phi$ in $C^1(\rpluscl;
  H^1(\Omega))$ of the wave equation model \eqref{WaveEq} satisfying
  the regularity conditions \eqref{WaveEqRegularityAssumptions}.  The
  output, given by \eqref{WaveEqObs}, satisfies $y \in
  C(\rpluscl;L^2(\Gamma(0)))$.
\end{proposition}
\noindent This is a reformulation of \cite[Theorem~5.1(ii) and
  Corollary~5.2]{A-L-M:AWGIDDS}. The analogous statements hold for
Webster's equation by \cite[Theorem~4.1]{A-L-M:AWGIDDS} but this will
not be needed in this article. It is remarkable that the current
approach does not leave any ``regularity gap'' in the sense that all
Sobolev smoothness allowed by \eqref{WaveEqRegularityAssumptions} is
actually used in the proof of Proposition~\ref{DirichletTraceProp}
below.


\subsubsection*{Global coordinates in a curved  tube}

We give next the necessary facts about the coordinate system we use to
represent tubular domains as shown in Fig.~\ref{centrelines}; see also
\cite[Section~2.3]{SR:WHER}. Detailed arguments are postponed to the 
appendices at the end of the article.

 \begin{figure}
 \centering
 \def\svgwidth{0.45\textwidth}
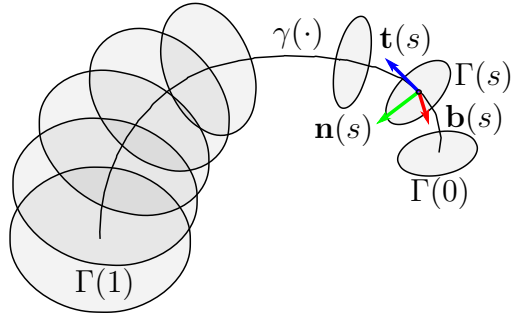 
 \caption{The Frenet frame of the planar centreline for a tubular
   domain $\Omega$, represented by some of its intersection surfaces
   $\Gamma(s)$ for $s \in [0,1]$. The wall $\Gamma \subset \partial \Omega$ is not
   shown. }\label{centrelines}
 \end{figure}

Let $\bg : [0,1] \to \R^3$ be a smooth curve parameterised by arc
length $s$.  Define for $s \in [0,1]$ the unit vectors
\begin{equation*}
  \bt (s) := \bg'(s), \quad \bn (s) := \frac{\bt'(s)}{\kappa(s)}, \quad
\text{ and } \quad \bb (s) := \bt (s) \times \bn (s).
\end{equation*}
Here the curvature of $\bg$ is defined by $\kappa(s) :=
\norm{\bg''(s)}$ for $s \in [0,1]$, and we assume that $\min_{s \in
  [0,1]}{\kappa(s)} > 0$.  It then follows that
\begin{equation} \label{FrenetDerivatives}
\bt' = \kappa \bn , \quad \bn' = - \kappa \bt  +  \tau \bb,  
\quad  \text{ and }  \quad  \bb' = - \tau \bn
\end{equation}
where the number $\tau = \tau(s)$ for $s \in [0,1]$ is called the
torsion of $\bg$.  All these facts concerning the Fren\'et coordinates
can be found in any standard reference on elementary differential
geometry, e.g., \cite{AP:EDG}.

By $\bi$, $\bj$, and $\bk$ denote the unit vectors in the direction of
cartesian coordinate axes. For any point $\br = x \bi + y \bj + z \bk
\in \R^3$ in a neighbourhood of $\bg$ we have
\begin{equation} \label{CoordinateIdentity}
  x \bi + y \bj + z \bk = \bg(s) + r \cos{\theta} \, \bn (s) + r \sin{\theta} \, \bb (s),
\end{equation}
and hence we can use the triple $(s,r,\theta) \in [0,1] \times
[0,\infty] \times [0, 2 \pi)$ as coordinates for $\br$. We shall call
these coordinates $(s, r, \theta)$ the \emph{tube coordinates}.

Using these coordinates, we can define the (open) interior of any tube
as
\begin{equation} \label{OmegaDef}
  \Omega := \{(s, r, \theta) : s\in (0,1), r \in [0, R(s)), \theta \in 
  [0,2 \pi) \} \subset \R^3
\end{equation}
where the differentiable function $R:[0,1] \to (0,\infty)$ gives the
radius of the cross-section at $s \in [0,1]$.  We define
\begin{equation} \label{GammaDef}
  \begin{cases}
    \Gamma & := \{(s, R(s), \theta) : s\in (0,1), \theta \in [0,2 \pi) \} \quad \text{ and } \\
    \Gamma(s) & :=\{(s, r, \theta) : r \in [0,R_0), \theta \in [0,2 \pi)
  \} \quad \text{ for } \quad s \in [0,1]. 
\end{cases}
\end{equation}
As before, we denote by $A(s)$ the area of $\Gamma(s)$ for $s \in
[0,1]$.  We call $\Gamma$ the \emph{wall}, and the circular plates
$\Gamma(0)$, $\Gamma(1)$ are the \emph{ends} of the tube $\Omega$. The
boundary of $\Omega$ satisfies $\partial \Omega = \closure{\Gamma}
\cup \Gamma(0) \cup \Gamma(1)$.  We make it as a standing assumption
that
\begin{equation} \label{NoFolding}
  \eta(s) := R(s) \kappa(s) < 1 \quad \text{ for all } \quad s \in [0,1]
\end{equation}
in which case the tube does not fold onto itself, and the coordinate
mapping associated to \eqref{CoordinateIdentity} is a diffeomorphism.
The number $\eta(s)$ is called \emph{curvature ratio} at $s \in
[0,1]$, and typically both $\eta(s) \ll 1$ and $\eta'(s) \ll 1$.
The \emph{curvature factor} is defined by
\begin{equation}\label{CurvatureFactor}
\Xi = \Xi(s,r,\theta) :=  \frac{1}{1 - r \kappa(s) \cos{\theta}}.
\end{equation}
Clearly $\Xi = 1 + r \kappa \cos{\theta} + \mathcal O(\eta^2)$ as
$\eta \to 0$ since $\abs{\Xi - 1 - r \kappa \cos{\theta}} \leq
\eta^2(1 - \eta)^{-1}$.

For the derivation of Webster's model, we need the following facts
about the coordinate system given by \eqref{CoordinateIdentity} that
can be verified by elementary vector calculus.

\begin{enumerate}
\item The volume differential in tube coordinates is
\begin{equation} \label{VolumeDifferential}
 d V = dx\, dy \, dz = \frac{r }{\Xi}  \, ds \, dr \, d\theta.
\end{equation}

\item The surface area element on the tube wall $\Gamma$ is 
  \begin{equation} \label{SurfaceAreaElement}
  dS= W \, d\theta \, ds  \quad \text{ where } 
  \quad W(s) :=  R(s)\sqrt{R'(s)^2+(\eta(s) - 1)^2}
\end{equation}

\item The gradient $\nabla$ can be written as 
\begin{equation}
  \label{NablaInTubeEq}
 \nabla = \bt(s) D_1 + \bn(s) D_2 + \bb(s) D_3 
\end{equation}
where the differential operators $D_i$, $i=1,2,3$, are given by
\begin{equation} \label{DiffOps}
  D_1 := \Xi \frac{\partial}{\partial s}, \quad 
   D_2 := \cos{\theta} \frac{\partial}{\partial r} -
  \frac{\sin{\theta}}{r}\frac{\partial}{\partial \theta}, \quad D_3 :=
  \sin{\theta} \frac{\partial}{\partial r} +
  \frac{\cos{\theta}}{r}\frac{\partial}{\partial \theta}. 
\end{equation}

\item Finally, the exterior normal derivative on $\Gamma$ is given by
  \begin{equation} \label{ExtNormalOnGamma}
    \bnu \cdot \nabla 
    = \frac{R}{W} \left (- R' \Xi \frac{\partial}{\partial s}
         + (1 - \kappa
      R)\frac{\partial}{\partial r} \right ).
  \end{equation}

\end{enumerate}

\section{Approximation of the wave equation \label{ResultSec}}

In this section we present how the solutions of the wave equation 
\begin{equation} \label{OnlyWaveEq}
 \frac{\partial^2 \phi}{\partial t^2}
= c^2 \Delta \phi \quad \text{ on the domain } \quad  \Omega \times \rplus 
\end{equation}
are related to a generalised Webster's equation in the domain $\Omega$
that has been described in \eqref{OmegaDef}.  We assume that the tube
wall $\Gamma$, defined by \eqref{GammaDef}, allows energy dissipation
through it, and the energy loss is modelled by
\begin{equation} \label{DampingBndry}
  \alpha \frac{\partial \phi}{\partial t} +\frac{\partial \phi }{\partial \bnu} = 0 
  \quad \text{ on } \quad  \Gamma  \times  \rplus \quad  \text{ where }  \quad \alpha \geq  0.
\end{equation}
Note that the Neumann boundary condition $\frac{\partial \phi
}{\partial \bnu} = 0$, i.e., $\alpha = 0$ in \eqref{DampingBndry},
describes a perfectly reflecting boundary leading to a lossless model.
Another interesting choice is $\alpha = 1/c$ which corresponds to a
perfectly absorbing boundary.  For the purpose of reprensenting the
curvature of the tube, we define the \emph{sound speed correction
  factor} by
\begin{equation} \label{SoundSpeedCorrectionFactor}
  \Sigma(s) := \left ( 1 + \tfrac{1}{4} \eta^2(s) \right )^{-1/2}
  \quad \text { where } \quad \eta(s) 
  \quad \text{ is given by \eqref{NoFolding}} . 
\end{equation} 
As shown in Sections \ref{WebsterSec} and \ref{ProofSection}, Webster's
equation is then given by
\begin{equation} \label{WebsterHorn}
  \frac{1}{A(s)} \frac{\partial}{\partial s} \left ( A(s)
    \frac{\partial \psi }{\partial s} \right ) - \frac{2 \pi
    \alpha W(s)}{A(s)} \frac{\partial \psi }{\partial t} -
  \frac{1}{c^2 \Sigma(s)^{2}} \frac{\partial^2 \psi}{ \partial
    t^2} = 0.
\end{equation}
Note that the boundary dissipation \eqref{DampingBndry} results in the
second term in \eqref{WebsterHorn}.

On the end surface $\Gamma(1)$ of tube $\Omega$ we use the Dirichlet
boundary condition
\begin{equation} \label{DirichletBndry}
   \phi = 0 \quad \text{ on } \quad  \Gamma(1) \times \rplus 
\end{equation}
that represents an acoustic open end. The same argument that produces
\eqref{WebsterHorn} from \eqref{OnlyWaveEq} -- \eqref{DampingBndry}
also produces the boundary condition $\psi(1,t) = 0$ from
\eqref{DirichletBndry}.

Recall that the wave equation \eqref{OnlyWaveEq} is controlled and
observed from the boundary component $\Gamma(0)$ representing one end
of the tube $\Omega$.  In this work, we use the scattering boundary
conditions given already in \eqref{WaveEq}--\eqref{WaveEqObs}:
\begin{equation} \label{WaveEquationInputBCResultSec}
\begin{aligned} 
  & c \frac{\partial \phi}{\partial \nu}(\br, t) + \frac{\partial \phi}{\partial t}(\br, t) =
  2 \sqrt{\tfrac{c}{\rho A(0)}}
  \, u(\br, t) \quad \text{ and } \\
  & c \frac{\partial \phi}{\partial \nu}(\br, t) - \frac{\partial \phi}{\partial t}(\br, t) =
  2 \sqrt{\tfrac{c}{\rho A(0)}} \, y(\br, t) 
\end{aligned}
\end{equation}
The treatment of these boundary conditions is somewhat more
complicated than that of \eqref{DirichletBndry}, and we make two
additional assumptions on $\Omega$ to ensure that the surface
$\Gamma(0)$ is ``geometrically reflection-free'':
\begin{equation} \label{ReflectFreeBoundaryAss}
  A'(0) = 0 \quad \text{ and } \quad \kappa(0) = 0.
\end{equation}
As is shown in Sections \ref{WebsterSec} and \ref{ProofSection}, then
\eqref{WaveEquationInputBCResultSec} imply the scattering boundary
conditions
\begin{equation} \label{WebsterEquationBoundaryCondition} 
\begin{aligned}
  &  - c  \frac{\partial
    \psi}{\partial s}(0,t) +  \frac{\partial \psi}{\partial t}(0,t) = 2 \sqrt{\tfrac{c}{\rho A(0)}} \, \bar u(t)
  \quad \text{ and } \\
  &  - c  \frac{\partial
    \psi}{\partial s}(0,t) - \frac{\partial \psi}{\partial
    t}(0,t) = 2 \sqrt{\tfrac{c}{\rho A(0)}} \, \bar y(t) 
\end{aligned}
\end{equation}
for Webster's equation. Other kinds of boundary conditions, even
nonlinear ones, on $\Gamma(0)$ are possible, and their treatment
follows the lines of Section \ref{WebsterSec}. 


In Fig.~\ref{Fig1}, both the wave equation and Webster's equation are
boundary controlled by a common external signal, apart from averaging
over $\Gamma(0)$. To produce an exact cancellation of outputs in the
upper right corner of Fig.~\ref{Fig1}, we need to directly control the
state of Webster's equation by the forcing functions $F$, $G$, and $H$
\begin{align} \label{ControlTermsF}
  F(s,t) & := - \frac{1}{A(s)} \frac{\partial}{\partial s} \left (
  A'(s) \left ( \bar \phi(s) - \frac{1}{2 \pi}\int_{0}^{2 \pi}
  \phi(s,R(s),\theta) \, d \theta \right ) \right ); \\
\label{ControlTermsG}
  G(s,t) & :=  - \frac{2 \pi  \alpha W(s)}{A(s)} \frac{\partial }{\partial t} \left ( \bar \phi(s) - \frac{1}{2
      \pi}\int_{0}^{2\pi}{ \phi(s,R(s),\theta)
      d\theta } \right ); \text{ and } \\
 \label{ControlTermsH}
  H(s,t) & := \int_{\Gamma(s)}{\frac{1}{\Xi} \nabla \left (
      \frac{1}{\Xi} \right )\cdot \nabla \phi \, d A } - \frac{1}{
    A(s)} \int_{ \Gamma(s) } {E \Delta \phi d A }
  \\
  &  -
  \frac{\alpha W(s) \eta(s)}{A(s)} \left ( \int_{0}^{2\pi}{
      \frac{\partial \phi}{\partial t} (s,R(s),\theta) \cos{\theta}
      d\theta } \right ) \nonumber
\end{align}
where the \emph{error function} is defined by
\begin{equation} \label{SSCFError}
   E(s,r,\theta) := \frac{1}{\Xi^2} - \frac{1}{\Sigma(s)^{2}}
 =  - 2 r \kappa(s) \cos{\theta} + \kappa(s)^2 ( r^2 \cos^2{\theta} - R(s)^2/4).
\end{equation}

Now we are ready to give the main result of this article whose proof is
in Section \ref{ProofSection}:
\begin{theorem}
  \label{MainTheorem1}
  Assume that $\phi: \Omega \times \rpluscl  \to \R$ is a
  solution of the wave equation \eqref{OnlyWaveEq} satisfying the
  regularity properties \eqref{WaveEqRegularityAssumptions} and the
  boundary conditions \eqref{DampingBndry} and \eqref{DirichletBndry}.
  \begin{enumerate}
  \item \label{MainTheorem1Claim1} The averaged solution $\bar
    \phi:(0,1) \times  \rpluscl \to \R$, defined by
    \begin{equation} \label{AveragedWaveEqSolution}
      \bar\phi=\frac{1}{A(s)}\int_{\Gamma(s)}\phi dA,
    \end{equation}
    satisfies the regularity conditions
    \begin{equation}     \label{AveragedSolutionRegular}
      \begin{aligned}
        &  \bar \phi \in 
        C^2(\rpluscl; L^2(0,1)) \quad \text{ and } \quad
        \frac{\partial \bar \phi}{\partial s} \in C^1(\rpluscl;L^2(0,1)), 
      \end{aligned}
    \end{equation}
  \item \label{MainTheorem1Claim2} The forcing terms $F = F(s,t)$, $G
    = G(s,t)$, and $H = H(s,t)$, as defined in \eqref{ControlTermsF}
    -- \eqref{ControlTermsH}, satisfy $F \in C(\rpluscl;L^2(0,1))$ and
    $G, H \in C^1(\rpluscl;L^2(0,1))$. 
  \item \label{MainTheorem1Claim3} The averaged solution $\bar \phi$ is a
    weak solution of
    \begin{equation} \label{WebsterEq1}  
      \begin{aligned}
        & \frac{1}{A(s)} \frac{\partial}{\partial s} \left ( A(s)
          \frac{\partial \bar \phi }{\partial s} \right ) - \frac{2 \pi
          \alpha W(s)}{A(s)} \frac{\partial \bar \phi }{\partial t}
        -  \frac{1}{c^2 \Sigma(s)^{2}} \frac{\partial^2 \bar \phi}{ \partial t^2} \\
        & = F + G + H. 
      \end{aligned}    
    \end{equation}
  \item \label{MainTheorem1Claim4} If $\phi$ satisfies the
    control/observation boundary conditions
    \eqref{WaveEquationInputBCResultSec} and $R'(0) = \kappa(0) = 0$,
    then $\bar \phi$ satisfies the corresponding boundary conditions
    \eqref{WebsterEquationBoundaryCondition}.
  \end{enumerate}
\end{theorem}
Claim \eqref{MainTheorem1Claim3} means plainly that
\begin{equation}
  \label{WebsterWeakForm}
  \begin{aligned}
    &\int_0^T\int_0^1\left(-\frac{\partial \bar\phi}{\partial
        s}\frac{\partial \zeta}{\partial
        s}-\frac{1}{c^2\Sigma(s)}\frac{\partial^2\bar\phi }{\partial
        t^2}\zeta\right)A(s) \, dsdt-  2\pi\alpha \int_0^T\int_0^1 
        W(s)\frac{\partial \bar\phi}{\partial t}\zeta
        dsdt\\ &=\int_0^T\int_0^1(F(s,t)+G(s,t)+H(s,t))\zeta(s,t)
        A(s) \, dsdt 
      \end{aligned}
    \end{equation}
    for all test functions $\zeta\in C_0^\infty((0,1)\times (0,T))$
    and all $T>0$.  Claims \eqref{MainTheorem1Claim3} and
    \eqref{MainTheorem1Claim4} state that the feed-forward connection
    in Fig.~\ref{Fig1} is well-defined, and the ``tracking error''
    output on the right vanishes.  By claims
    \eqref{MainTheorem1Claim1} and \eqref{MainTheorem1Claim2}, the
    solution and all signals in \eqref{WebsterEq1} are functions
    rather than distributions whenever the conditions of
    Proposition~\ref{PieniPorpositioRatkeavuudestaKaava1} are
    satisfied.

    Note that Theorem~\ref{MainTheorem1} does not claim that $\bar
    \phi$ would satisfy \eqref{WebsterHorn}. By substracting the
    Webster's equations on top of Figs.~\ref{Fig1}--\ref{Fig2}, we see
    that the difference $\bar \phi - \psi$ satisfies a non-homogenous
    Webster's equation \eqref{WebsterEq1} with zero boundary control
    at $s=0$, leading to error estimates to be given in
    \cite{L-M:PEEWEWP}.

\section{From wave equation to  Webster's equation \label{WebsterSec}}

In this section, we derive Webster's equation from the wave equation.
We postpone the detailed rigorous justification of these computations
under the assumptions \eqref{WaveEqRegularityAssumptions} until
Section~\ref{ProofSection}.  The reader may find it useful to think of
$\phi$ as a function in $C^\infty(\overline{\Omega})$, albeit this is
of course too good to be true for a general solution of the wave
equation.


\subsubsection*{Derivation of Webster's equation}

To derive Webster's equation \eqref{WebsterEq1}, we first obtain a
weak formulation of the wave equation \eqref{OnlyWaveEq} --
\eqref{DampingBndry} in terms of tube coordinates.  Open portions of
the tube $\Omega$ are defined by
 \begin{equation*}
  \Omega(s_0,s_1) = \{ (s,r,\theta) : s \in
  (s_0, s_1), r \in [0, R(s)), \theta \in [0, 2 \pi) \} 
 \end{equation*}
 for $0 \leq s_0 < s_1 \leq 1$. We denote $\Gamma(s_0,s_1) := \Gamma
 \cap \clos{\Omega(s_0,s_1)}$.  On the cross-sectional surface
 $\Gamma(s_0)$ we have $\bnu = -\bt(s_0)$, and similarly we have $\bnu
 = \bt(s_1)$ on $\Gamma(s_1)$.  Thus by 
 \eqref{CurvatureFactor} and \eqref{NablaInTubeEq} we get for the exterior normal derivatives
\begin{equation} \label{EndPlaneBnryOp} 
\begin{aligned} 
\frac{\partial }{\partial \bnu}  & = - \Xi \frac{\partial}{\partial s} 
   \quad \text{ for  } \quad (s_0, r, \theta) \in \Gamma(s_0) \quad \text{ and }  \\
    \frac{\partial }{\partial \bnu} 
 &  =  \Xi \frac{\partial}{\partial s} 
     \quad \text{ for } \quad (s_1, r, \theta) \in \Gamma(s_1).
\end{aligned}
\end{equation}
Clearly $dA = r dr d\theta$ on both $\Gamma(s_0)$ and $\Gamma(s_1)$.

We use a version of Green's identity given in
\cite[Theorem~A.3]{A-L-M:AWGIDDS} for tubular Lipschitz domains, valid
for $\phi$ satisfying regularity requirements
\eqref{WaveEqRegularityAssumptions}.  Together with
\eqref{EndPlaneBnryOp} we obtain
\begin{align*} 
  & \int_{\Omega(s_0,s_1)}{ \Xi^{-1} \Delta \phi \, d V} \\ 
  & = - \int_{\Omega(s_0,s_1)}{\nabla \left ( \frac{1}{\Xi} \right )\cdot   \nabla \phi  \, d V} 
  + \int_{\Gamma(s_0,s_1) }{ \frac{1}{\Xi} \frac{\partial \phi}{\partial \bnu} d A} + 
\int_{\Gamma(s_0) \cup \Gamma(s_1) }{ \frac{1}{\Xi} \frac{\partial \phi}{\partial \bnu} d A} \\
   & =  - \int_{\Omega(s_0,s_1)}{\nabla \left ( \frac{1}{\Xi} \right ) \cdot \nabla \phi  \, d V}
- \alpha \int_{\Gamma(s_0,s_1) }{ \frac{1}{\Xi}  \frac{\partial \phi}{\partial t}  d A} \\
  &  \hspace{5cm} + \int_{ \Gamma(s_1)}{\frac{\partial \phi }{\partial s}  d A}
  - \int_{\Gamma(s_0)}{ \frac{\partial \phi }{\partial s}  d A} 
\end{align*}
where $\int_{\Gamma(s_0,s_1) }{ \Xi^{-1} \frac{\partial
    \phi}{\partial \bnu} d A} = - \alpha \int_{\Gamma(s_0,s_1) }{
  \Xi^{-1} \frac{\partial \phi}{\partial t} d A}$ is implied by
\eqref{DampingBndry}.  
On the other hand, the regularity requirements
\eqref{WaveEqRegularityAssumptions} include the fact that $\Delta\phi$
is an $L^2$ function. Thus $\phi$ satisfies the wave equation
\eqref{OnlyWaveEq} pointwise almost everywhere.  Using this
observation and the volume element \eqref{VolumeDifferential}, we get
\begin{align*}
   \int_{\Omega(s_0,s_1)}{ \Xi^{-1} \Delta \phi \, d V}
  & = \int_{\Omega(s_0,s_1)}{\frac{1}{c^2 \Xi} \frac{\partial^2 \phi}{\partial t^2} \, d V} 
   = \int_{s_0}^{s_1} {\left ( \int_{0}^{2 \pi} \int_{0}^{R(s)} {
        \frac{ 1}{c^2 \Xi^2}
        \frac{\partial^2 \phi}{\partial t^2}
        r\, dr\, d\theta } \right ) \, ds } \\
  & = \int_{s_0}^{s_1}{   \left (
      \int_{ \Gamma(s)}{\frac{1}{c^2 \Xi^2} \frac{\partial^2 \phi}{\partial t^2}  d A} \right ) \, ds}.
\end{align*}
We combine these two expressions for the same thing, and get
\begin{equation} \label{WaveEqAveraged}
  L(s_0,s_1)  =  \int_{\Omega(s_0,s_1)}{\nabla \left ( \frac{1}{\Xi} \right )\cdot   \nabla \phi  \, d V} 
\end{equation}
where
\begin{equation}
\label{WaveEqAveragedRHS}
\begin{aligned} 
  L(s_0,s_1) & := \overbrace{\int_{ \Gamma(s_1) }{\frac{\partial \phi }{\partial s}  d A}
  - \int_{\Gamma(s_0)  }{ \frac{\partial \phi }{\partial s}  d A}}^{\text{(i)}} \\
& -  \alpha \underbrace{\int_{\Gamma(s_0,s_1) }{ \frac{1}{\Xi}  \frac{\partial \phi}{\partial t}  d A}}_{\text{(ii)}} - \underbrace{ \int_{s_0}^{s_1} { 
    \left ( \int_{ \Gamma(s) }{\frac{1}{c^2 \Xi^2} \frac{\partial^2 \phi}{\partial t^2}  d A} \right    ) \, ds}}_{\text{(iii)}}. 
\end{aligned}
\end{equation}
We note that \eqref{WaveEqAveraged} -- \eqref{WaveEqAveragedRHS} hold
for any solution $\phi$ of the wave equation satisfying the regularity
conditions \eqref{WaveEqRegularityAssumptions}.
\begin{remark} 
  Using \eqref{CurvatureFactor} -- \eqref{NablaInTubeEq} we see that
\begin{equation*} 
      \nabla \left ( \frac{1}{\Xi} \right ) 
= - \bt r \kappa' \Xi \cos{\theta}  - \bn \kappa.
\end{equation*}
We conclude that in the case of uncurved tube $\kappa = 0$, the right
hand side of \eqref{WaveEqAveraged} vanishes, and the we get
$L(s_0,s_1) = 0$ for all $s_0, s_1 \in (0,1)$. The equation
$L(s_0,s_1) = 0$ can be regarded as an integrated form of Webster's
model.  
\end{remark}

We aim at deriving the loaded Webster's equation \eqref{WebsterEq1}
for the averaged solution $\bar\phi$ given by
\eqref{AveragedWaveEqSolution}.  To achieve this in the general case
for curved tubes, we first study the limit
\begin{equation}\label{LimitOfEL}
  \lim_{s' \to  s}\frac{L(s,s')}{s' - s};
\end{equation}
here $L(s,s')$ is given by \eqref{WaveEqAveragedRHS}. A rigorous
interpretation of this limit will be given in the proof of
Theorem~\ref{MainTheorem1}.  Then we express the outcome of the limit
process in terms of the averaged solution $\bar\phi$.  Since the
averaged solution $\bar\phi$ solves Webster's equation only
approximately, we obtain a non-vanishing load term $f$ in
\eqref{WebsterEq1}.


\subsubsection*{Terms (i) in \eqref{WaveEqAveragedRHS}}
The first two terms in \eqref{WaveEqAveragedRHS} give the elliptic
part of Webster's equation.  We rewrite
 $ \frac{\partial }{\partial s}\left(\int_{\Gamma(s)}\phi dA\right)=
  \frac{\partial }{\partial s}\left(A(s)\bar\phi\right)$
as
\begin{align*}
   A(s) \frac{\partial \bar \phi }{\partial s} & = - \frac{2
    R'(s)}{R(s)} {\int_{ \Gamma(s) }{\phi d A}}
  + \frac{\partial }{\partial s} \left ( {\int_{\Gamma(s)}{\phi d A}} \right )  \\
  & = -  A'(s) \bar \phi +
  \frac{\partial}{\partial s} \left ( {\int_{\Gamma(s)}{\phi d A}}
  \right ). 
\end{align*}
The last term is the limit as $s' \to s$ of
\begin{align*}
  & \frac{1}{s' - s} \left (\int_{\Gamma(s')}{\phi dA} - \int_{\Gamma(s)}{\phi dA} \right )  \\
  & = \int_{0}^{2 \pi} \int_{0}^{R(s')}\frac{\phi(s',r,\theta) -
    \phi(s,r,\theta)}{s' - s} r \, dr \, d\theta  
  \\
  & + \frac{R(s') - R(s)}{s' - s} \int_{0}^{2 \pi}\left (
    \frac{1}{R(s') - R(s)}
    \int_{R(s)}^{R(s')} \phi(s,r,\theta) r \, dr  \right ) d \theta
\end{align*}
which gives 
\begin{equation*}
  \frac{\partial }{\partial s}\left(\int_{\Gamma(s)}\phi dA\right)=
  \int_{\Gamma(s) }{ \frac{\partial \phi }{\partial s} d A} 
  +  R(s) R'(s) \int_{0}^{2 \pi} \phi(s,R(s),\theta) \, d \theta.
\end{equation*}
From these expressions, we now obtain 
\begin{equation*}
  \int_{\Gamma(s)}\frac{\partial \phi}{\partial s}dA=
  A'(s)\left(\bar\phi-\frac{1}{2\pi}\int_0^{2\pi}
    \phi(s,R(s),\theta)d\theta\right)
  +A(s)\frac{\partial \bar\phi}{\partial s} .
\end{equation*}
Initially this holds for all $\phi \in C^\infty(\overline{\Omega})$
and $s \in (0,1)$. After an application of
Proposition~\ref{BOperatorProp} and Proposition~\ref{AOperatorProp} we
see that this is true also for $\phi$ satisfying
\eqref{WaveEqRegularityAssumptions} for almost every $s$.  Thus
\begin{align} \label{SpaceDerivativeLimit}
& \lim_{s' \to s} {\frac{1}{s'-s}\left ( \int_{ \Gamma(s') }{\frac{\partial \phi }{\partial s}  d A}
  - \int_{\Gamma(s)  }{ \frac{\partial \phi }{\partial s}  d A} \right )} =  
 \frac{\partial}{\partial s}\int_{\Gamma(s)  }{ \frac{\partial \phi }{\partial s}  d A}  \\
&   =  \frac{\partial}{\partial s}\left(A(s) 
  \frac{\partial \bar \phi }{\partial s}\right)
  + \frac{\partial}{\partial s}\left(A'(s) \left ( \bar \phi(s)  - \frac{1}{2 \pi}\int_{0}^{2 \pi}  
    \phi(s,R(s),\theta) \,  d \theta \right )\right) \nonumber
\end{align}
The reader can recognise the elliptic term of Webster's equation on
the last line. The assumptions \eqref{WaveEqRegularityAssumptions} are
not strong enough to give the averaged solution $\bar\phi$ two weak
derivatives with respect to $s$.  Hence we will interpret the last
equation in the sense of distributions in the proof of
Theorem~\ref{MainTheorem1}.


\subsubsection*{Term (ii) in \eqref{WaveEqAveragedRHS}}
For the boundary dissipation term in \eqref{WaveEqAveragedRHS} we get
\begin{align*}
  & -  \alpha \int_{\Gamma(s,s') }{ \frac{1}{\Xi}  \frac{\partial \phi}{\partial t}  d A}
= -  \alpha \frac{\partial }{\partial t} \int_{s}^{s'} 
  \left (   \int_{0}^{2\pi}{ \phi(\tilde s,R(\tilde s),\theta) \frac{ d\theta }{\Xi} } \right )  W(\tilde s)  d\tilde s 
\end{align*}
by using \eqref{SurfaceAreaElement}.
Thus
\begin{equation}\label{DissipationTermEq1}
  \begin{aligned}
  &   \lim_{s' \to s}{\frac{-  \alpha}{s' - s} \int_{\Gamma(s,s') }{ \frac{1}{\Xi}  \frac{\partial \phi}{\partial t}  d A}} \\
  & = - \alpha \lim_{s' \to s} \left [ {\frac{1}{s' - s} }
    \int_{s}^{s'}
    \left (   \int_{0}^{2\pi}{ \frac{\partial \phi}{\partial t}(\tilde s,R(\tilde s),\theta)  \frac{ d\theta }{\Xi} } \right )  W(\tilde s)  d\tilde s \right ] \\
  & = - \alpha W(s) \frac{\partial }{\partial t} \left (
    \int_{0}^{2\pi}{ \phi(s,R(s),\theta) \frac{ d\theta }{\Xi} }
  \right ) .
\end{aligned}
\end{equation}
We will use Proposition~\ref{AverageConvProp} to justify this limit.

To proceed we note that by \eqref{CurvatureFactor}, the expression
inside the parenthesis takes the form
\begin{align*}
 \int_{0}^{2\pi}{ \phi(s,R(s),\theta) \frac{ d\theta }{\Xi}}   = & 
  \int_{0}^{2\pi} {\phi(s,R(s),\theta)
    (1 - \eta(s) \cos{\theta} ) d\theta } \\
   =  & 2 \pi \bar \phi(s) - 2 \pi \left ( \bar \phi(s) - \frac{1}{2
      \pi}\int_{0}^{2\pi}{ \phi(s,R(s),\theta)
      d\theta } \right ) \\
  & - \eta(s) \int_{0}^{2\pi}{ \phi(s,R(s),\theta) \cos{\theta} d\theta }.
\end{align*}
Thus we conclude that 
\begin{align}
  \label{DissipationTermEq}
  & \lim_{s' \to s}{\frac{-  \alpha}{s' - s} \int_{\Gamma(s,s') }{
      \frac{1}{\Xi} 
      \frac{\partial \phi}{\partial t}  d A}}
  = -  2 \pi \alpha W \frac{\partial \bar \phi }{\partial t} \nonumber  \\
  & +  2 \pi  \alpha W \frac{\partial }{\partial t} 
  \left ( \bar \phi(s) - \frac{1}{2
      \pi}\int_{0}^{2\pi}{ \phi(s,R(s),\theta)
      d\theta } \right ) \\
  & 
  + \alpha W \eta  \frac{\partial }{\partial t} 
  \left (\int_{0}^{2\pi}{ \phi(s,R(s),\theta) \cos{\theta} d\theta } 
  \right ). \nonumber
\end{align}

\subsubsection*{Term (iii) in \eqref{WaveEqAveragedRHS}}
We take the familiar limit also in the final term in
\eqref{WaveEqAveragedRHS}. We get
\begin{align} \label{TimeDerivativeLimit}
  & \lim_{s' \to s}{\frac{1}{s' - s} \int_{s}^{s'} { \left ( \int_{
          \Gamma(s) }{\frac{1}{c^2 \Xi^2} \frac{\partial^2
            \phi}{\partial t^2} d A} \right ) \, ds}} = \frac{1}{c^2}
  \int_{ \Gamma(s) }{\frac{1}{\Xi^2} \frac{\partial^2 \phi}{\partial
      t^2} d A} \\ & = \frac{1}{c^2} \frac{\partial^2 }{ \partial t^2}
  \left ( \int_{ \Gamma(s) }{\frac{\phi d A}{\Xi^2} } \right ).
  \nonumber
\end{align}
We will appeal to Proposition~\ref{AverageConvProp} to deal with this
limit.  In order to obtain \eqref{WebsterEq1}, we must express most of
the contribution of the term containing $\int_{ \Gamma(s) }{\Xi^{-2}
  \phi \, dA }$ in \eqref{TimeDerivativeLimit} using the averaged
solution $\bar \phi$.  Unfortunately, the curvature factor $\Xi$ is
not constant, and we cannot just bring it out from under the integral
sign.  To deal with this problem, we use the sound speed correction
factor $\Sigma(s)$. It is clear that $\Sigma(s)^{-2}$ is the best
estimate for function $\Xi^{-2}$ over $\Gamma(s)$ in the sense of
least squares.  We have
\begin{equation} \label{SoundSpeedCorrectionFactorDerived}
\begin{aligned}
  & \frac{1}{\Sigma(s)^{2}} := \frac{1}{A(s)}
  \int_{ \Gamma(s) }{\frac{d A}{\Xi^2}}  \\
  & = \frac{1}{\pi R(s)^2} \int_{0}^{2 \pi}\int_0^{R(s)}{(1 - r
    \kappa(s) \cos{\theta})^2 \, r dr d\theta} = 1 + \tfrac{1}{4}
  \eta^2(s) 
\end{aligned}
\end{equation}
 where the curvature ratio $\eta(s)$ is given by \eqref{NoFolding}.
 The error function in \eqref{SSCFError} satisfies the identity $\Xi E
 = \Xi^{-1} - \Xi -\tfrac{1}{4}\Xi \eta^2$.  With the aid of this
 we get the splitting
\begin{equation*}
  \int_{ \Gamma(s) }{\frac{\phi d A}{\Xi^2} } = \frac{
  A(s)}{\Sigma(s)^{2}} \bar \phi + \int_{ \Gamma(s) }{E \phi d A }
\end{equation*}
which implies
\begin{equation} \label{TimeDerivativeLimitCont}
   \frac{1}{c^2} \frac{\partial^2 }{ \partial t^2} \left ( \int_{\Gamma(s) }{\frac{\phi d A}{\Xi^2} } \right ) 
   = \frac{ A(s)}{c(s)^2} \frac{\partial^2 \bar \phi}{ \partial t^2}
   + \int_{ \Gamma(s) }{ \frac{E}{c^2} \frac{\partial^2 \phi}{ \partial t^2} \,dA}
\end{equation}
where
\begin{equation} \label{VariableSoundSpeedDef}
 c(s) := c \Sigma(s)
\end{equation}
is the speed of sound after correction by the curvature factor; see
\eqref{SoundSpeedCorrectionFactor}.

Taking the limit on the right hand side of \eqref{WaveEqAveraged}, we
get
\begin{equation} \label{RHSTerm2} 
 \lim_ {s' \to s}
  {\frac{1}{s' - s}\int_{\Omega(s,s')}{\nabla \left ( \frac{1}{\Xi}
      \right )\cdot \nabla \phi \, d V}} =
  \int_{\Gamma(s)}{\frac{1}{\Xi} \nabla \left (
      \frac{1}{\Xi} \right )\cdot \nabla \phi \, d A } .
\end{equation}
Proposition~\ref{AverageConvProp} again justifies this limit.



We now have everything we need for the limit \eqref{LimitOfEL}.  We
put \eqref{WaveEqAveraged}, \eqref{SpaceDerivativeLimit} --
\eqref{TimeDerivativeLimit}, \eqref{TimeDerivativeLimitCont} and
\eqref{RHSTerm2} together, and obtain
\begin{equation} \label{WebsterRSHLimitEq}
  \begin{aligned}
    0=&\frac{1}{A(s)} \lim_{s' \to s}{\frac{L(s,s')}{(s' - s)}}
  -\frac{1}{A(s)}\int_{\Gamma(s)}{\frac{1}{\Xi} \nabla \left (
      \frac{1}{\Xi} \right )\cdot \nabla \phi \, d A } \\
    = & \frac{1}{A(s)} \frac{\partial}{\partial s} \left 
      (\int_{\Gamma(s)}{ \frac{\partial
          \phi }{\partial s} d A} \right )  
    - \frac{\alpha W(s)}{A(s)} \frac{\partial }{\partial t} 
    \left (\int_{0}^{2\pi}{ \phi(s,R(s),\theta) \frac{ d\theta }{\Xi} }\right )\\
    &-  \frac{1}{A(s)c^2} \frac{\partial^2 }{ \partial t^2} 
    \left ( \int_{ \Gamma(s) }{\frac{\phi d A}{\Xi^2} } \right )
    -\frac{1}{A(s)}\int_{\Gamma(s)}{\frac{1}{\Xi} \nabla \left (
      \frac{1}{\Xi} \right )\cdot \nabla \phi \, d A }\\  
     = &
    \frac{1}{A(s)} \frac{\partial}{\partial s} \left ( A(s)
      \frac{\partial \bar \phi }{\partial s} \right )
    -  \frac{2 \pi \alpha W(s)}{A(s)}  \frac{\partial \bar \phi }{\partial t}
    -  \frac{1}{c^2 \Sigma(s)^{2}} \frac{\partial^2 \bar \phi}{ \partial t^2}  \\
    & -  F(s,t) - G(s,t)-H(s,t)
  \end{aligned}
\end{equation}
where $F$, $G$, and $H$ are given by \eqref{ControlTermsF} --
\eqref{ControlTermsH}. Equation \eqref{WebsterRSHLimitEq} is Webster's
equation \eqref{WebsterEq1}. If we assume that the right hand side of
\eqref{WebsterEq1} is negligible, we obtain the Webster's horn
equation \eqref{WebsterHorn} with curvature and dissipation.


\subsubsection*{Webster's control/observation boundary conditions at $s = 0$}

We next derive the boundary conditions
\eqref{WebsterEquationBoundaryCondition} for Webster's equation at $s
= 0$ that correspond to the boundary conditions of the wave equation
on $\Gamma(0)$ in \eqref{WaveEquationInputBCResultSec}.  

We get for smooth $\phi$ by using \eqref{EndPlaneBnryOp}
\begin{equation} \label{ScatteringBoundaryEq1}
\begin{aligned}
   \int_{\Gamma(0)}{\left (c \frac{\partial \phi}{\partial \nu} \pm
      \frac{\partial \phi}{\partial t} \right )dA} & = - c \int_{\Gamma(0)}{\frac{\partial
      \phi}{\partial s}\, dA} \pm \frac{\partial}{\partial t}
  \int_{\Gamma(0)}{\phi \, dA} \\
  & - c  \int_{\Gamma(0)}{\frac{\partial \phi}{\partial s}
       \left (\Xi
      - 1 \right )\, dA}
\end{aligned}
\end{equation}
since $\int_{\Gamma(0)}{\frac{\partial \phi}{\partial \theta} \, dA} =
0$.  The last term on the right hand side of
\eqref{ScatteringBoundaryEq1} vanishes due to the assumption
$\kappa(0) = 0$. Thus, our interest lies in the term
\begin{equation*}
  \int_{\Gamma(0)}{\frac{\partial \phi}{\partial s}\, dA}
  = \int_{0}^{2 \pi} \int_0^{R(0)} { \frac{\partial \phi}{\partial s}(s,r,\theta) r dr d\theta} \quad  \text{ for }
  \quad s = 0.
\end{equation*}
We get
\begin{align*}
  & \int_{0}^{2 \pi} \int_0^{R(0)} { \frac{\partial \phi}{\partial s}(s,r,\theta) r dr d\theta} \\
  = & \frac{\partial }{\partial s} \int_{0}^{2 \pi} {\left (
      \int_0^{R(s)} { \phi (s,r,\theta) r dr }
      + \int_{R(s)}^{R(0)} { \phi (s,r,\theta) r dr } \right ) d\theta} \\
  = & \frac{\partial ( A \bar \phi) }{\partial s} - \frac{\partial
  }{\partial s} \left (\int_{0}^{2 \pi} \int_{R(0)}^{R(s)} { \phi
    (s,r,\theta) r dr d\theta} \right ).
\end{align*}
For the latter term on the right hand side we get
\begin{align*}
  & \int_{0}^{2 \pi} \left ( \frac{\partial }{\partial s} 
    \int_{R(0)}^{R(s)} { \phi (s,r,\theta) r dr } \right ) d\theta \\
  = & \int_{0}^{2 \pi} {\left ( \lim_{s' \to s}{ \int_{R(0)}^{R(s')} {
          \frac{\phi (s',r,\theta)-\phi (s,r,\theta)}{s' - s} r dr }}
      + \lim_{s' \to s}{\frac{1}{s' - s}
        \int_{R(s)}^{R(s')} { \phi (s,r,\theta) r dr }} \right ) d\theta} \\
  = & \int_{0}^{2 \pi} \int_{R(0)}^{R(s)} { \frac{\partial \phi}{\partial s}
    r dr d\theta} + R(s)R'(s) \int_{0}^{2 \pi}{\phi(s,R(s),\theta) \, d \theta}
\end{align*}
and hence by setting $s = 0$ above we obtain
\begin{equation*}
  \int_{\Gamma(0)}{\frac{\partial \phi}{\partial s}\, dA} 
  =  \frac{\partial (A \bar \phi) }{\partial s} \left (0 \right )
  - R(0)R'(0) \int_{0}^{2 \pi}{\phi(0,R(0),\theta) \, d \theta}.
\end{equation*}
We conclude from this and \eqref{ScatteringBoundaryEq1} that
\begin{equation} \label{WebsterScatteringBoundaryCondition}
  \frac{1}{A(0)} \int_{\Gamma(0)}{\left (c \frac{\partial
        \phi}{\partial \nu} \pm \frac{\partial \phi}{\partial t} \right )dA} = - c \frac{\partial
    \bar \phi}{\partial s}(t,0) \pm  \frac{\partial \bar \phi}{\partial t}(t,0)  - K(t) 
\end{equation}
for all $t > 0$ where
\begin{equation} \label{BoundaryTermDisturbation}
  K(t) : = \frac{c A'(0)}{A(0)} \left (\bar \phi(t,0) -  
    \frac{1}{2 \pi} \int_0^{2 \pi}{ \phi(t,0,R(0),\theta)\, d \theta} \right ). 
\end{equation}
Now, if $\phi$ satisfies \eqref{WaveEquationInputBCResultSec} then
$\bar \phi$ satisfies
\begin{equation} \label{WebsterEquationInputOutputBoundaryCondition}
\begin{aligned}
  &  - c(0) \frac{\partial
    \bar \phi}{\partial s}(0,t) +  \frac{\partial \bar \phi}{\partial t}(0,t)  - K(t) = 2 \sqrt{\tfrac{c}{\rho A(0)}} \, \bar u(t) \\
  &  - c(0) \frac{\partial
    \bar \phi}{\partial s}(0,t) - \frac{\partial \bar \phi}{\partial
    t}(0,t) - K(t) = 2 \sqrt{\tfrac{c}{\rho A(0)}} \, \bar y(t) 
\end{aligned}
\end{equation}
where $\bar u(t) : = \frac{1}{A(0)} \int_{\Gamma(0)}{u(\br,t) dA}$.
The assumption $\kappa(0) = 0$ is convenient here, too, since it
implies $c(0) = c$ where $c(s) := c \Sigma(s)$ is the variable sound
speed \eqref{VariableSoundSpeedDef}.

Thus, for smooth $\phi$ we have \eqref{WaveEquationInputBCResultSec}
$\Rightarrow$ \eqref{WebsterEquationBoundaryCondition} if $\kappa(0) =
A'(0) = 0$ which we make a standing assumption in \cite{L-M:PEEWEWP};
see also \cite[p.~1992]{SR:WHER}. Passage from smooth functions $\phi$
to those satisfying \eqref{WaveEqRegularityAssumptions} is
accomplished by a straightforward density argument.

\section{\label{ProofSection} Proof of Theorem \ref{MainTheorem1}}

In this section, we give the detailed proof of the main result based
on the computations in the previous section. Some auxiliary results
are needed first.

\begin{proposition} \label{DirichletTraceProp}
  Assume that $\Omega$ is a tubular domain described by
  \eqref{OmegaDef} and $\phi:\Omega \times \rpluscl \to \R$ satisfies
  the regularity assumptions $\phi \in C^1(\rpluscl; H^1(\Omega))$ and
  $\Delta \phi \in C(\rpluscl; L^2(\Omega))$ in
  \eqref{WaveEqRegularityAssumptions}. By $\Gamma \subset \partial
  \Omega$ denote the walls of the tube.  Then the boundary trace
  satisfies $\phi \rst{\Gamma} \in C(\rpluscl;H^{1}(\Gamma))$.
\end{proposition}
\begin{proof}
  Let $\psi \in H^1(\Omega)$ such that $\Delta \psi \in L^2(\Omega)$.
 Since the radius function $R(s)$ is smooth, we may assume that the
 tube $\Omega$ is of constant diameter near its ends $\Gamma(0)$ and
 $\Gamma(1)$ --- if not, use a diffeomorphims to obtain $R'(s) = 0$
 for $s \in (0,\eta) \cup (1 - \eta, 1)$ for $\eta > 0$.  We extend
 $\Omega$ from both ends to a longer tubular open set $\tilde \Omega$
 that has a smooth boundary. This extension can be carried out in many
 ways but the function $\psi$ must be extended to all of $\tilde
 \Omega$ so that \eqref{WaveEqRegularityAssumptions} are satisfied by
 the extended function $\tilde \psi$. Let us consider the end $s = 0$.

For $s \in (-\eta/2, 1)$ define the extension by reflection
\begin{equation}
  \tilde \psi(s, r, \theta) = 
  \begin{cases}
    \psi(s, r, \theta) & \text{ for } s \in (0,1) \\
    3 \psi(-s, r, \theta)  -2 \psi(-2s, r, \theta) & \text{ for } s \in (-1/2,0) \\
  \end{cases}
\end{equation}
and similarly at the other end. This extension gives us a function
$\tilde \psi$ defined on a tube of length $1 + \eta$, and by smoothing
the tube around the ends produces $\tilde \Omega$ having a smooth
boundary with $\Omega \subset \tilde \Omega$ and $\Gamma \subset
\partial \tilde \Omega$. It is easy to see that $\tilde \psi$
satisfies \eqref{WaveEqRegularityAssumptions} since $\psi$ does.
 
Thus $\tilde \psi \in H^{1}(\tilde \Omega)$ with $\Delta \tilde \psi
\in L^2(\tilde \Omega)$ and $\frac{\partial \tilde\psi}{\partial \nu} \in
L^2(\partial \tilde\Omega)$ where the boundary $\partial \tilde \Omega$ is
smooth.  By elliptic regularity theory (see \cite[Remark 7.2 on
p. 188]{L-M:NHBVPAPI}, with $s = 3/2$, $m = 1$, $m_0 = 1$), we
conclude that $\tilde \psi \in H^{3/2}(\tilde \Omega)$.  Using
\cite[Theorem 7.4]{L-M:NHBVPAPI} we conclude that $\tilde \psi \in
D^{3/2}_\Delta(\tilde \Omega) := \{f \in H^{3/2}(\tilde \Omega) :
\Delta f \in \Xi^{-1/2}(\Omega) \}$ where $\Xi^{k}(\Omega)$ has been
defined on \cite[p. 172]{L-M:NHBVPAPI} since $L^2(\tilde \Omega)
\subset \Xi^{-1/2}(\Omega)$. It follows from \cite[Theorem 7.4 on
p. 188]{L-M:NHBVPAPI} that $\tilde \psi\rst{\partial \tilde \Omega}
\in H^{1}(\partial \tilde \Omega)$ and thus $\psi\rst{\Gamma} \in
H^{1}(\Gamma)$ by restriction. Now, the trace mapping $\psi \mapsto
\psi\rst{\Gamma}$ is bounded from $E(\Delta; L^2(\Omega)) := \{ f \in
H^1(\Omega) : \Delta f \in L^2(\Omega) \}$ into $L^2(\Omega)$, and its
range is in $H^1(\Gamma)$; it is thus bounded from $E(\Delta;
L^2(\Omega))$ into $H^1(\Gamma)$.  Since $\phi \in C(\rpluscl; E(\Delta;
L^2(\Omega)))$, we conclude that $\phi\rst{\Gamma} \in C(\rpluscl;
H^1(\Gamma))$.
\end{proof}

\begin{proposition} \label{BOperatorProp} Let $\Gamma$ be the wall of
  the tube, as in Proposition \ref{DirichletTraceProp}.  Define for $g
  \in H^{1}(\Gamma)$ the linear mapping $\mathcal B$ by
\begin{equation*}
  (\mathcal B g)(s) := \frac{1}{2 \pi}\int_{0}^{2 \pi} g(s,R(s),\theta)
\, d \theta \quad \text{for}\quad s \in (0,1).
\end{equation*}
Then 
\begin{equation*}
  \mathcal B \in\BLO(H^{k}(\Gamma);H^{k}(0,1))\quad\text{for all}\quad
  k \in \R.
\end{equation*}
\end{proposition}
\begin{proof}
   For any $g \in
  C^{\infty}(\Gamma)$ and $w \in L^2(0,1)$, we have $\mathcal B g \in
  C^{\infty}(0,1)$ and
  \begin{align*}
    & \left <\mathcal B g, w \right >_{L^2(0,1)} = \frac{1}{2 \pi}  \int_0^1 {\left ( 
      \int_0^{2 \pi}{g(s,R(s),\theta) \, d \theta} \right ) w(s) \, ds } \\
    &  = \int_0^1 \int_0^{2 \pi} {  g(s,R(s),\theta)  \frac{ w(s)}{2 \pi W(s)} \, dA  }
    = \left < g ,   \tilde w   \right >_{L^2(\Gamma)}
  \end{align*}
  where $d A = W(s) d \theta ds$ and $\tilde w(s,\theta) = \frac{
    w(s)}{2 \pi W(s)}$. Because $C^{\infty}(\Gamma)$ is dense in
  $H^{1}(\Gamma)$, we conclude that $\mathcal B^* \in
  \BLO(H^{-k}(0,1);H^{-k}(\Gamma))$ for all $k \geq 0$ where $\left
  (\mathcal B^* w \right )(s, R(s), \theta) = \frac{w(s)}{2 \pi W(s)}$
  for all $w \in L^2(0,1))$ and $(s, R(s), \theta) \in \Gamma$.  In
  particular, $\mathcal B \in \BLO(H^{k}_0(\Gamma);H^{k}_0(0,1))$ for
  all $k \geq 0$.

  Define the spaces 
  \begin{equation*}
    W := \{ r \in C(\Gamma) : r(s,R(s), \theta) = s r_0(\theta) +
    (1 - s) r_1(\theta) \text{ for } r_0, r_1 \in C(\T) \}
  \end{equation*}
  and $W^k := W \cap H^k(\Gamma)$ equipped with the norm of
  $H^k(\Gamma)$. Then $W^k$ is a closed subspace of $H^k(\Gamma)$ for
  $k \geq 2$, and it is easy to see that $W^k \dot + H^k_0(\Gamma) =
  H^k(\Gamma)$. The operator $\mathcal B$ maps $W^k$ boundedly onto
  linear functions on $(0,1)$ that are equipped with the norm of
  $H^k(0,1)$. We conclude that $\mathcal B \in
  \BLO(H^{k}(\Gamma);H^{k}(0,1))$ for $k \geq 2$ and $k = 0$. The
  claim follows from this by interpolation; see, e.g., 
  \cite[Theorem 5.1 on p. 27 and Theorem 7.7 on p. 36]{L-M:NHBVPAPI}. 
\end{proof}

\begin{proposition} \label{AOperatorProp}
Let $\Omega$ be as in Proposition \ref{DirichletTraceProp}.  For $f
\in L^2(\Omega)$, define the linear mapping $\mathcal A$ by 
\begin{equation*}
  (\mathcal Af )(s) := \frac{1}{A(s)} {\int_{ \Gamma(s) }{\phi d
      A}}\quad\text{for}\quad s\in (0,1).  
\end{equation*}
Then
\begin{equation*}
  \mathcal A \in 
  \BLO(H^k(\Omega); H^k(0,1)) \quad\text{for}\quad k\geq 0.
\end{equation*}

\end{proposition}
\begin{proof}
 For $\phi \in L^2(\Omega)$ we have by H\"older's inequality
  \begin{align*}
    \norm{\mathcal A \phi}^2_{L^2(0,1)} & = \int_{0}^1{\left (
      \frac{1}{A(s)} {\int_{ \Gamma(s) }{\phi d A}} \right )^2 \, ds}
    \\ & \leq \int_{0}^1{ \frac{1}{A(s)^2} \left (
      \int_{\Gamma(s)}{\phi^2 d A} \right ) \left ( \int_{\Gamma(s)}{
        d A} \right ) \, ds} \\ & \leq \int_{0}^1\int_{\Gamma(s)}{
      \frac{\Xi(s)}{A(s)} \phi^2 \cdot \frac{r dr d \theta ds}{\Xi(s)}
    } \leq C_1 \norm{\phi}^2_{L^2(\Omega)}
  \end{align*}
where $C_1 := \max_{s \in [0,1]} {\frac{\Xi(s)}{A(s)}}$
  because $dA = r dr d \theta$ on $\Gamma(s)$ and $dV$ is given by
  \eqref{VolumeDifferential}. Thus $\mathcal A \in \BLO(L^2(\Omega); L^2(0,1))$.

  Using the operators $\mathcal A$ and $\mathcal B$, equation
  \eqref{SpaceDerivativeLimit} takes the form
  \begin{equation} \label{ABIntertwiningRelation}
    \frac{\partial}{ \partial s} \left ( \mathcal A \phi \right ) =
    \mathcal A \left ( \frac{\partial \phi}{ \partial s} \right ) +
    \frac{A'(s)}{A(s)} \left ( \mathcal B (\phi\rst{\Gamma}) -
    \mathcal A (\phi) \right )
  \end{equation}
  first for $\phi \in C^\infty(\overline{\Omega})$, and then by
  density for all $\phi \in H^1(\Omega)$ and $\phi\rst{\Gamma} \in
  H^{1/2}(\Gamma)$ and \cite[Theorem 1.5.1.3]{PG:EPNSD}.  It follows
  that $\norm{\frac{\partial}{ \partial s} \left ( \mathcal A \phi
    \right )}_{L^2(0,1)} \leq C_1 \norm{\bt \cdot \nabla
    \phi}_{L^2(0,1)} + \sup_{s \in [0,1]}{\frac{A'(s)}{A(s)}} \left (
    C_2 \norm{\phi\rst{\Gamma}}_{L^2(\Gamma)} + C_1
    \norm{\phi}_{L^2(0,1)} \right ) \leq C \norm{\phi}_{H^1(\Omega)}$
  for all $\phi \in H^1(\Omega)$ for some $C < \infty$. This means
  that $\mathcal A \in \BLO(H^1(\Omega);H^1(0,1))$, and the claim
  follows by interpolation, as in the previous proposition.
\end{proof}

We treat some of the limits as $s'\to s$ in Section \ref{WebsterSec}
by Proposition¨\ref{AverageConvProp}. In the proof, we use the
centered Hardy--Littlewood maximal operator, defined for functions
$f\in L^1_{\mathrm{loc}}(\R)$ by setting
\begin{equation*}
  \left ( Mf \right )(x)=\sup_{h>0}\frac{1}{2h}\int_{x-h}^{x+h}\abs{f}\,dx.
\end{equation*}
It is well-known that the non-linear operator $M$ is bounded from
$L^p(\R)$ to $L^p(\R)$ for $1<p\leq \infty$; see, e.g.,
\cite[Theorem~2.5 on p.~31]{JD:FA} or \cite[Theorem~8.18; in
  particular Eq.~(6) on p.~174]{WR:RCA}.
 
\begin{proposition}\label{AverageConvProp}
  Let $\tilde{f}\in L^2(\Omega)$, $\tilde{g}\in L^2(\Gamma)$,
  and define for any $h>0$  the functions
  \begin{align*}
    f(s)=&\int_{\Gamma(s)} \tilde{f}\,dA=A(s) \left (\mathcal A \tilde{f}  \right )(s) \quad \text{ for } \quad  s\in [0,1],\\
    f_h(s)=&
     \begin{cases}
      \frac{1}{h}\int_{\Omega(s+h,s)} \Xi \tilde{f}\, dV & \quad \text{ for } \quad s\in [h,1-h]; \\
      0 & \quad \text{otherwise;}
    \end{cases}\\
    g(s)=&\int_{0}^{2\pi}\tilde{g}(s,R(s),\theta)\,d\theta
    =2\pi(\mathcal B\tilde{g})(s),\quad s\in [0,1], \\
    g_h(s)=&
    \begin{cases}
      \frac{1}{h}\int_{\Gamma(s+h,s)} W^{-1} \tilde{g}\, dS & \quad \text{ for } \quad s\in [h,1-h];\\
      0 & \quad \text{otherwise.}
    \end{cases}
  \end{align*}
  Then  
$f_h\to f$ and $ g_h\to g$ pointwise Lebesgue a.e. as well as in
  $L^2(0,1)$ as $h\to 0$.
\end{proposition}
\begin{proof}
  We prove only the claim concerning $f_h$ as the case of $g_h$ is
  essentially identical. That $f\in L^2(0,1)$ follows from
  Proposition~\ref{AOperatorProp}.  For $h\leq s\leq 1-h$
  \begin{equation*}
    f_h(s)=\frac{1}{h}\int_{s}^{s+h}\int_{\Gamma(\sigma)} \tilde{f}\,dA\,d\sigma
=\frac{1}{h}\int_{s}^{s+h}\int_{\Gamma(\sigma)} \Xi \tilde{f} \cdot \Xi^{-1}dA\,d\sigma
    =\frac{1}{h}\int_{s}^{s+h} f(\sigma)\,d\sigma
  \end{equation*}
  by Fubini's theorem, recalling that $dV = \Xi^{-1} dA \, ds$ by
  \eqref{VolumeDifferential}.
By Lebesgue's theorem (see, e.g., \cite[Corollary~2.13]{JD:FA} or
\cite[Theorems~7.7 and 7.10]{WR:RCA}), we have $\abs{f(s)} \leq
(Mf)(s)$ a.e. $s \in [0,1]$, and 
  \begin{equation*}
    \abs{ f_h(s)}\leq \sup_{h>0} \abs{f_h(s)}
    \leq 2 \cdot \sup_{h>0}\frac{1}{2h}
    \int_{s-h}^{s+h}\abs{f(\sigma)}\,d\sigma= 2 \left ( Mf \right )(s),
  \end{equation*}
 Thus, 
\begin{equation} \label{AverageConvPropEq1}
  \abs{f_h(s) - f(s)}^2 \leq 9  \left ( Mf \right )(s)^2 \text{ for a.e. } s \in [0,1]
\end{equation}
where the upper bound is in $L^1(0,1)$ since the Hardy-Littlewood
maximal operator $M$ maps $L^2(0,1)$ into itself.  By the Lebesgue's
theorem, the left hand side of \eqref{AverageConvPropEq1} converges to
zero pointwise a.e. on $[0,1]$. Hence, $f_h \to f$ in $L^2(0,1)$ by
the Lebesgue Dominated Convergence theorem
\cite[Theorem~1.34]{WR:RCA}.
\end{proof}

We are now ready to give a rigorous proof for Theorem
\ref{MainTheorem1}.

\begin{proof}[Proof of Theorem \ref{MainTheorem1}]
  Claim \eqref{MainTheorem1Claim1}: By
  \eqref{WaveEqRegularityAssumptions} we have $\phi \in
  C^2(\rpluscl;L^2(\Omega))$ and $ \frac{\partial \phi}{\partial s}\in
  C^1(\rpluscl;L^2(\Omega)) $. Thus an application of Propositions
  \ref{AOperatorProp} and \ref{BOperatorProp} implies
  that $\bar \phi \in C^2(\rpluscl;L^2(0,1))$ and
  $\frac{\partial \bar \phi}{ \partial s} \in C^1(\rpluscl;L^2(0,1))$,
  as desired.

  Claim \eqref{MainTheorem1Claim2}: The functions $A(s), A(s)^{-1},
  \eta(s), W(s), \Xi^{-1} \nabla (\Xi^{-1})$, and $E$ are all smooth
  by assumptions. Hence, by inspection of formulae
  \eqref{ControlTermsF} -- \eqref{ControlTermsH} we need to show, in
  addition to claim \eqref{MainTheorem1Claim1} of this theorem, that
  $\frac{\partial }{\partial s} \left ( \mathcal B \phi\rst{\Gamma}
  \right ) \in C(\rpluscl;L^2(0,1))$ in order to prove that $F \in
  C(\rpluscl;L^2(0,1))$ and $G, H \in C^1(\rpluscl;L^2(0,1))$.  This
  follows directly from Propositions \ref{DirichletTraceProp} and
  \ref{BOperatorProp}.

  Claim \eqref{MainTheorem1Claim3}: 

  Let $0<h<1$. For $h<s<1-h$ we may write, in the notation of
  Section~\ref{WebsterSec},
  \begin{equation}\label{MainTheorem1Proof1}
    L(s,s+h)=\int_{\Omega(s+h,s)}\nabla
    \left(\frac{1}{\Xi}\right)\cdot \nabla \phi dV,
  \end{equation}
  where 
  \begin{align*}
    L(s,s+h)=&\int_{\Gamma(s+h)} \frac{\partial \phi}{\partial
      s}dA-\int_{\Gamma(s)}
    \frac{\partial \phi}{\partial s}dA\\
    &-\alpha\int_{\Gamma(s+h,s)}\frac{1}{\Xi}\frac{\partial
      \phi}{\partial
      t}dA-\int_s^{s+h}\int_{\Gamma(\sigma)}\frac{1}{c^2\Xi^2}
    \frac{\partial^2  \phi}{\partial t}dAd\sigma.
  \end{align*}
  This makes sense for solutions of the wave equation satisfying
  \eqref{WaveEqRegularityAssumptions}: the wave equation holds
  pointwise almost everywhere, and Green's formula
  \cite[Theorem~A.3]{A-L-M:AWGIDDS} applies as well.

  We want to derive the weak form of Webster's equation. To this end,
  take an arbitrary test function $\zeta\in C^\infty_0((0,1)\times
  (0,T))$, and choose $h$ sufficiently small, so that the spatial
  support of $\zeta$ is contained in the interval
  $(3h,1-3h)$.  By a change of variables, we see that
  \begin{equation}\label{MainTheorem1Proof2}
    \int_0^1\int_{\Gamma(s+h)}\frac{\partial \phi}{\partial
      s}dA \, \zeta(s,t)ds=\int_0^1\int_{\Gamma(s)}\frac{\partial \phi}{\partial
      s}dA \, \zeta(s-h)ds.
  \end{equation}
  Next we multiply \eqref{MainTheorem1Proof1} by $\zeta$, integrate
  over $s$, divide by $h$ and use \eqref{MainTheorem1Proof2}. This
  leads to
  \begin{equation}
    \label{MainTheorem1Proof3}
    \begin{aligned}
      \overbrace{\int_0^1\frac{1}{h} \int_{\Omega(s+h,s)}\nabla
      \left(\frac{1}{\Xi}\right)\cdot \nabla \phi \zeta(s,t)dV\,ds}^{I_h}
    &=\overbrace{\int_0^1\int_{\Gamma(s)}\frac{\partial\phi}{\partial
        s}dA \frac{\zeta(s-h,t)-\zeta(s,t)}{h}ds}^{II_h}\\
      &
      -\alpha\underbrace{\int_0^1\frac{1}{h}\int_{\Gamma(s+h,s)}\frac{1}{\Xi}
        \frac{\partial\phi}{\partial t}dS \, \zeta(s,t)ds}_{III_h}\\
      &-
      \underbrace{\int_0^1\frac{1}{h}\int_s^{s+h}
        \int_{\Gamma(\sigma)}\frac{1}{c^2\Xi^2}\frac{\partial
        ^2\phi}{\partial t^2}dA \, d\sigma \, \zeta(s,t)ds}_{IV_h}.
    \end{aligned}
  \end{equation}
    We next take the limit of all the terms in
  \eqref{MainTheorem1Proof3} as $h$ tends to zero. For the term $I_h$,
  we use \eqref{MainTheorem1Proof1} and Proposition
  \ref{AverageConvProp} with $\tilde{f} :=\Xi^{-1} \nabla \Xi^{-1} \cdot\nabla
  \phi $ to get
  \begin{equation*}
    I:=\lim_{h\to 0}I_h=\int_0^1\int_{\Gamma(s)}
    { \frac{1}{\Xi} \nabla \left(\frac{1}{\Xi}\right)\cdot\nabla \phi\,dA\,\zeta(s,t)ds }
= \int_{\Omega} {\left [ \nabla \left(\frac{1}{\Xi}\right)\cdot\nabla \phi \right ] \,\zeta  \,dV}.
  \end{equation*}
  The limit of the term $II_h$ on the right is
  \begin{align*}
    II:=&\lim_{h\to 0} II_h= \lim_{h\to 0}
    \int_0^1\int_{\Gamma(s)}\frac{\partial\phi}{\partial
      s}dA \frac{\zeta(s-h,t)-\zeta(s,t)}{h}ds\\ =&-\int_0^1  \left ( \int_{\Gamma(s)}
    \frac{\partial\phi}{\partial s}dA \right ) \frac{\partial\zeta(s,t)}{\partial
      s}ds
  \end{align*}
  since we have a difference quotient on the smooth test function
  $\zeta$.  We handle the other two terms on the right by an
  application of Proposition~\ref{AverageConvProp} with $g =
  \frac{W}{\Xi} \frac{\partial\phi}{\partial t}\rst\Gamma$ to obtain
  \begin{align*}
    III :=&\lim_{h\to 0}III_h=
    \lim_{h\to 0}\int_0^1\frac{1}{h}\int_{\Gamma(s+h,s)}\frac{1}{\Xi}\frac{\partial  \phi}{\partial t}\,dS \,\zeta(s,t)ds\\ 
    = &\int_0^1 \int_{0}^{2\pi} W \left ( \frac{1}{\Xi} \frac{\partial\phi}{\partial t} \right )\rst\Gamma\,d\theta\,\zeta(s,t)ds; \text{ and }\\
    IV:=&\lim_{h\to 0} IV_h=\lim_{h\to 0}\int_0^1\int_s^{s+h}\int_{\Gamma(\sigma)}\frac{1}{c^2\Xi^2}
    \frac{\partial^2  \phi}{\partial t}\,dA\,d\sigma \, \zeta(s,t)\,ds\\=&
    \int_0^1 \int_{\Gamma(s)}\frac{1}{c^2\Xi^2}
    \frac{\partial^2  \phi}{\partial t}\,dA \, \zeta(s,t)\,ds
  \end{align*}
  because $dS = W d \theta ds$ on $\Gamma$ by
  \eqref{SurfaceAreaElement}.

  The proof is completed by expressing the limiting terms $II$, $III$,
  and $IV$ above in terms of $\bar\phi$. For term $II$, we use
  \eqref{ABIntertwiningRelation} to handle the derivative with respect
  to $s$, and get
  \begin{align*}
    II=&-\int_0^1A(s)\frac{\partial \bar\phi}{\partial s}(s)
    \frac{\partial\zeta}{\partial s}(s,t)ds\\
    &+\int_0^1\frac{1}{A(s)}\frac{\partial}{\partial s}
    \left(A'(s)(\mathcal A\phi(s)-\mathcal B(\phi\rst\Gamma))\right)
    \zeta(s,t)A(s)\,ds.
  \end{align*}
  after integrating by parts in the second term as well. This yields the
  forcing term $F(s,t)$.

 For term $III$, we note that passing from
  \eqref{DissipationTermEq1} to \eqref{DissipationTermEq} requires
  only adding and substracting suitable terms. The same is true for
  passing from \eqref{TimeDerivativeLimit} to
  \eqref{TimeDerivativeLimitCont}, and this takes care of the term
  $IV$. Hence we get
  \begin{align*}
    III=&2\pi\int_0^1\left( W\frac{\partial\bar\phi}{\partial t}- W
      \frac{\partial}{\partial t}\left(\bar\phi-\mathcal B(\phi\rst\Gamma)
      \right)
      -W\eta \frac{\partial}{\partial t}\mathcal B(\phi\rst\Gamma\cos\theta) 
    \right)\zeta(s,t)\,ds\\
    IV=&\int_0^1A(s)\left(\frac{1}{c(s)^2}
      \frac{\partial^2\bar\phi}{\partial t^2}+\frac{1}{A(s)}
      \int_{\Gamma(s)}\frac{E}{c^2}\frac{\partial^2\phi }{\partial t^2}\,dA\right)
    \zeta(s,t)\,ds.
  \end{align*}
  
  To finish off, we recall that passing to the limit $h\to 0$ in
  \eqref{MainTheorem1Proof3} gives the equation
  \begin{equation*}
    I=II-\alpha III-IV.
  \end{equation*}
  We insert all the above results into this equation, and get the weak
  form of \eqref{WebsterEq1} as desired.

  Claim \eqref{MainTheorem1Claim4}: As argued in Section
  \ref{WebsterSec}, we have \eqref{WaveEquationInputBCResultSec}
  $\Rightarrow$ \eqref{WebsterEquationBoundaryCondition} for smooth
  $\phi$ if $\kappa(0) = R'(0) = 0$. Since $\frac{\partial
    \phi}{\partial \nu}(\cdot, t) \in L^2(\Gamma(0))$ and
  $\frac{\partial \phi}{\partial t}\rst{\Gamma(0)}(\cdot, t) \in
  H^{1/2}(\Gamma(0))$ by \eqref{WaveEqRegularityAssumptions}, and the
  claim follows because the averaging over $\Gamma(0)$ is bounded
  linear functional on $L^2(\Gamma(0))$.
\end{proof}


\appendix


\section{Global coordinates for tubular domains}

\subsection*{Surface area element}

By \eqref{OmegaDef}, the wall $\Gamma$ can be parametrised as
\begin{equation*}
  \br(s,\theta)=\gamma(s)+R(s)(\bn\cos\theta+\bb\sin\theta)
\end{equation*}
where $s\in [0,1]$ and $\theta\in[0,2\pi]$. Hence the surface area
element can be found by using the formula
\begin{equation*}
  dS = \abs{\tfrac{\partial \br}{\partial
    s}\times \tfrac{\partial \br}{\partial \theta}}\, ds \, d\theta.
\end{equation*}

We have $\frac{\partial \br}{\partial \theta} =
R(s)(-\bn(s)\sin\theta+\bb(s)\cos\theta)$, and by
\eqref{FrenetDerivatives} we get
\begin{equation*}
 \frac{\partial \br}{\partial s} = 
 \bt(1-\kappa R)+\bn(R'\cos\theta-R\tau\sin\theta)
 +\bb(R'\sin\theta+R\tau\cos\theta).
\end{equation*}
Since $\{ \bt, \bn, \bb \}$ is a right-hand orthogonal system, we see
that
\begin{align*}
  \frac{\partial \br}{\partial s} \times\bn 
  = \bb(1-\kappa R)-\bt(R'\sin\theta+R\tau\cos\theta) 
\end{align*}
and
\begin{align*}
   \frac{\partial \br}{\partial s} \times\bb 
  =&-\bn(1-\kappa R)+\bt(R'\cos\theta-R\tau\sin\theta).
\end{align*}
Hence the exterior normal derivative lies in the direction of
\begin{equation} \label{UnnormalisedNormalVector}
  - \frac{\partial \br}{\partial s} \times \frac{\partial \br}{\partial \theta} 
  =  R[- \bt R' + \bn (1 - \kappa R) \cos\theta + \bb (1 - \kappa R) \sin\theta ]
\end{equation}
since $\kappa R < 1$.  Now \eqref{SurfaceAreaElement} follows by
orthonormality, see \eqref{NoFolding}.

\subsection*{Volume element}

To verify \eqref{VolumeDifferential} we need to compute the Jacobian
of the coordinate trasformation:
\begin{align*}
  & \frac{\partial (s, r, \theta)}{\partial (x, y, z)}
  = \left | \begin{matrix} 
      \frac{\partial s}{\partial x} &  \frac{\partial s}{\partial y} & \frac{\partial s}{\partial z} \\
      \frac{\partial r}{\partial x} &  \frac{\partial r}{\partial y} & \frac{\partial r}{\partial z} \\
      \frac{\partial \theta}{\partial x} &  \frac{\partial \theta}{\partial y} & \frac{\partial \theta}{\partial z} 
    \end{matrix}\right | \\
  & = 
  \left | 
    \begin{matrix} 
      \Xi (\bt \cdot \bi)  & \Xi (\bt \cdot \bj)    & \Xi (\bt \cdot \bk)   \\
      \left [ \cos{\theta} \, \bn + \sin{\theta} \, \bb\right ] \cdot  \bi   
      & \left [ \cos{\theta} \, \bn + \sin{\theta} \, \bb\right ] \cdot  \bj     
      & \left [ \cos{\theta} \, \bn + \sin{\theta} \, \bb\right ] \cdot  \bk     \\
      \left [ - \tau \Xi  \bt - \frac{\sin{\theta}}{r} \bn + \frac{\cos{\theta}}{r} \bb    \right ] \cdot \bi 
      & \left [ - \tau \Xi  \bt - \frac{\sin{\theta}}{r} \bn + \frac{\cos{\theta}}{r} \bb \right ] \cdot \bj 
      & \left [ - \tau \Xi  \bt -  \frac{\sin{\theta}}{r} \bn + \frac{\cos{\theta}}{r} \bb \right ] \cdot \bk
    \end{matrix}
  \right |  \\
  & = \Xi \cos{\theta} 
  \left | \begin{matrix} 
      \bt \cdot \bi  & \bt \cdot \bj  & \bt \cdot \bk   \\
      \bn  \cdot  \bi    &  \bn  \cdot  \bj      &  \bn  \cdot  \bk     \\
      \left [ - \tau \Xi  \bt -
        \frac{\sin{\theta}}{r} \bn + \frac{\cos{\theta}}{r} \bb
      \right ] \cdot \bi & \left [ - \tau \Xi  \bt -
        \frac{\sin{\theta}}{r} \bn + \frac{\cos{\theta}}{r} \bb
      \right ] \cdot \bj & \left [ - \tau \Xi  \bt -
        \frac{\sin{\theta}}{r} \bn + \frac{\cos{\theta}}{r} \bb
      \right ] \cdot \bk
    \end{matrix}
  \right |  
  \\ 
  & + \Xi \sin{\theta} \left | \begin{matrix} 
      \bt \cdot \bi  & \bt \cdot \bj  & \bt \cdot \bk   \\
      \bb \cdot  \bi   
      &    \bb \cdot  \bj     
      &    \bb \cdot  \bk     \\
      \left [ - \tau \Xi  \bt -
        \frac{\sin{\theta}}{r} \bn + \frac{\cos{\theta}}{r} \bb
      \right ] \cdot \bi & \left [ - \tau \Xi  \bt -
        \frac{\sin{\theta}}{r} \bn + \frac{\cos{\theta}}{r} \bb
      \right ] \cdot \bj & \left [ - \tau \Xi  \bt -
        \frac{\sin{\theta}}{r} \bn + \frac{\cos{\theta}}{r} \bb
      \right ] \cdot \bk
    \end{matrix}\right |  \\
  & = \frac{\Xi \cos^2{\theta}}{r} 
  \left | \begin{matrix} 
      \bt \cdot \bi  & \bt \cdot \bj  & \bt \cdot \bk   \\
      \bn  \cdot  \bi    &  \bn  \cdot  \bj      &  \bn  \cdot  \bk     \\
       \bb  \cdot \bi 
      &  \bb\cdot \bj 
      &  \bb  \cdot \bk
    \end{matrix}
  \right |  
   - \frac{\Xi \sin^2{\theta}}{r} \left | \begin{matrix} 
      \bt \cdot \bi  & \bt \cdot \bj  & \bt \cdot \bk   \\
      \bb \cdot  \bi & \bb \cdot  \bj &    \bb \cdot  \bk     \\
      \bn  \cdot \bi &  \bn   \cdot \bj &  \bn   \cdot \bk
    \end{matrix}\right |  
  = \frac{\Xi }{r} 
  \left | \begin{matrix} 
      \bt \cdot \bi  & \bt \cdot \bj  & \bt \cdot \bk   \\
      \bn  \cdot \bi &  \bn \cdot \bj & \bn \cdot \bk     \\
      \bb  \cdot \bi &  \bb \cdot \bj & \bb \cdot \bk
    \end{matrix}
  \right |.  
\end{align*}
We write the basis change as $x \bi + y \bj + z \bk = \alpha \bt +
\beta \bn + \gamma \bb $, and thus we see that
\begin{equation*}
  \left [ \begin{matrix} \alpha \\ \beta \\ \gamma  \end{matrix} \right ] = 
  A   \left [ \begin{matrix} x \\ y \\ z  \end{matrix} \right ]
\text{ where }
A := \left [ \begin{matrix} 
      \bt \cdot \bi  & \bt \cdot \bj  & \bt \cdot \bk   \\
      \bn  \cdot \bi &  \bn \cdot \bj & \bn \cdot \bk     \\
      \bb  \cdot \bi &  \bb \cdot \bj & \bb \cdot \bk
    \end{matrix}
  \right ].
\end{equation*}
Because both the bases $\{\bi , \bj , \bk \}$ and $\{\bt , \bn , \bb
\}$ are orthogonal, the basis change matrix $A$ is unitary. Hence its
determinant is of absolute value $1$, and $\abs{\frac{\partial (s, r,
    \theta)}{\partial (x, y, z)}} = \frac{\Xi}{r}$.  From this we
conclude that \eqref{VolumeDifferential} holds.

\subsection*{Gradient and normal derivative on $\Gamma$}

We have
\begin{equation*}
  \nabla = {\bf i}\frac{\partial }{\partial x} 
  + {\bf j} \frac{\partial }{\partial y} +
  {\bf k} \frac{\partial }{\partial z}.
\end{equation*}
Suppose that $\phi = \phi(s,r,\theta)$ where $s$, $r$, and $\theta$
are functions of $x, y$, and $z$.  We have by the chain rule $
\frac{\partial }{\partial x} = \frac{\partial s}{\partial x}
\frac{\partial }{\partial s} + \frac{\partial r}{\partial x}
\frac{\partial}{\partial r} + \frac{\partial \theta}{\partial x}
\frac{\partial}{\partial \theta}$, and hence we need to compute
$\frac{\partial s}{\partial x}$, $ \frac{\partial r}{\partial x}$, and
$\frac{\partial \theta}{\partial x}$.  By differentiating
\eqref{CoordinateIdentity} with respect to $x$, we get by using
\eqref{FrenetDerivatives}
\begin{align*} 
  \bi & = \left (1 - r \kappa \cos{\theta} \right ) \frac{\partial s}{\partial x}  \bt (s) \\
  & \left ( - r \tau \sin{\theta} \frac{\partial s}{\partial x} 
+ \cos{\theta} \frac{\partial r}{\partial x} - r
  \sin{\theta} \frac{\partial \theta}{\partial x} \right ) \bn (s) \\
  & + \left ( r \tau \cos{\theta} \frac{\partial s}{\partial x} 
+ \sin{\theta} \frac{\partial r}{\partial x} + r
  \cos{\theta} \frac{\partial \theta}{\partial x} \right ) \bb (s).
\end{align*}
Because $\bt (s)$, $\bn (s)$, and $\bb (s)$ are orthonormal, we get
\begin{equation*}
  \left [ \begin{matrix} 1 - r \kappa \cos{\theta} & 0 & 0 \\
      -r \tau \sin{\theta} & \cos{\theta} & -\sin{\theta} \\
      r \tau \cos{\theta} & \sin{\theta} & \cos{\theta} \\
    \end{matrix} \right ]
  \left [ \begin{matrix} 
      \frac{\partial s}{\partial x} \\ \frac{\partial r}{\partial x} 
      \\ r\frac{\partial \theta}{\partial x}  
    \end{matrix} \right ] = 
  \left [ \begin{matrix} \bt(s) \\  \bn(s) 
      \\  \bb(s) \end{matrix} \right ]   \cdot  \bi. 
\end{equation*}
By the topmost row, $\frac{\partial s}{\partial x} = \frac{\bt(s)}{1 -
  r \kappa \cos{\theta}} \cdot \bi$, and using this gives
\begin{equation*}
  \left [ \begin{matrix} 
      \frac{\partial r}{\partial x} 
      \\ r \frac{\partial \theta}{\partial x}  
    \end{matrix} \right ] =
  \left [ \begin{matrix} 
      \cos{\theta} & \sin{\theta} \\
      -\sin{\theta} & \cos{\theta} \\
    \end{matrix} \right ]
  \left [ \begin{matrix} 
      \bn(s) + \frac{r \tau \sin{\theta}}{1 - r \kappa \cos{\theta}} \bt(s) \\ 
      \bb(s) - \frac{r \tau \cos{\theta}}{1 - r \kappa \cos{\theta}} \bt(s) 
    \end{matrix} \right ]  \cdot  \bi 
\end{equation*}
or
\begin{equation*}
  \begin{cases}
    \frac{\partial s}{\partial x} & = \frac{\bt(s)}{1 - r \kappa \cos{\theta}} \cdot \bi \\
    \frac{\partial r}{\partial x} &
    = \left [ \cos{\theta} \, \bn(s) + \sin{\theta} \, \bb(s)\right ] \cdot  \bi  \\
    \frac{\partial \theta}{\partial x} & 
     = \left [ - \frac{\tau}{1 - r \kappa \cos{\theta}} \bt(s) -
      \frac{\sin{\theta}}{r} \bn(s) + \frac{\cos{\theta}}{r} \bb(s)
    \right ] \cdot \bi
  \end{cases}
\end{equation*}
We now conclude that $\frac{\partial \phi}{\partial x} = F \cdot \bi$ where   
\begin{align*} 
  F(s,r,\theta) &  := 
  \bt(s) \left (  \frac{\frac{\partial \phi}{\partial s} 
      - \tau \frac{\partial \phi}{\partial \theta} } {1 - r \kappa \cos{\theta}}  \right )
\\ &   + \bn(s) \left (\cos{\theta} \frac{\partial \phi}{\partial r}   -  
    \frac{\sin{\theta}}{r}\frac{\partial \phi}{\partial \theta}\right )
+ \bb(s) \left ( \sin{\theta} \frac{\partial \phi}{\partial r}   +  
    \frac{\cos{\theta}}{r}\frac{\partial \phi}{\partial \theta} \right ). \nonumber
\end{align*}
A similar argument shows that also $\frac{\partial \phi}{\partial y} =
F \cdot \bj$ and $\frac{\partial \phi}{\partial z} = F \cdot \bk$
hold, and hence $\nabla \phi = F$. The formulas \eqref{NablaInTubeEq}
and \eqref{DiffOps} can be read from this.

It remains to compute the normal derivative $\frac{\partial
}{ \partial \nu}$ on the tube wall $\Gamma$.  By
\eqref{UnnormalisedNormalVector}, \eqref{SurfaceAreaElement}, and
\eqref{NablaInTubeEq} we get
\begin{align*}
   \bnu \cdot \nabla 
  & = \frac{ R}{W} \left (- \bt R' + (1 - \kappa R) (\bn \cos\theta +
    \bb \sin\theta ) \right )
  \cdot (\bt D_1 + \bn D_2 + \bb  D_3 ) \\
  & = \frac{R}{W} \left (- R' \Xi \left ( \frac{\partial}{\partial s}
      - \tau \frac{\partial}{\partial \theta} \right ) + (1 - \kappa
    R)\frac{\partial}{\partial r} \right ),
\end{align*}
as desired.

\end{document}